\documentclass[11pt,reqno]{amsart}
\renewcommand{\baselinestretch}{1}

\usepackage{graphicx} 
\usepackage[dvipsnames]{xcolor}
\usepackage[british]{babel}
\usepackage{pgfplots}
 \pgfplotsset{compat = newest} 
 \usetikzlibrary{math}
 \usetikzlibrary{calc}
 \usetikzlibrary{shapes.geometric}
 \usetikzlibrary{shapes,positioning}
 \usetikzlibrary{arrows}

\usepackage{amsmath,amssymb,amsthm,comment,times,array,csquotes}
\usepackage[left=2.5cm,top=2.5cm,right=2.5cm,bottom=2.5cm]{geometry}
\usepackage[normalem]{ulem}
\usepackage{mathtools}
\usepackage[shortlabels]{enumitem}
\usepackage{xspace}
\usepackage{bbold}
\usepackage{bm}
\usepackage[backend=biber,giveninits=true, doi=false,url=false,isbn=false,style=trad-plain,hyperref]{biblatex}
\bibliography{references} 
\AtEveryBibitem{\clearlist{language}}

\usepackage[a-1b]{pdfx} 
\usepackage{hyperref}
\usepackage{thmtools}
\usepackage{cleveref}
\usepackage{crossreftools}

\newtheorem{theorem}{Theorem}[section] 
\newtheorem{lemma}[theorem]{Lemma}

\newtheorem{definition}[theorem]{Definition}
\newtheorem{strategy}[theorem]{Strategy}
\newtheorem*{metarule}{Metarule}
\newtheorem*{execution}{Execution}

\newtheorem{question}[theorem]{Question}

\Crefname{claim}{Claim}{Claims}

\theoremstyle{definition}

\theoremstyle{remark}
\newtheorem{remark}[theorem]{Remark}

\DeclareMathOperator{\Poisson}{Po}
\newcommand{\Poi}[1]{\Poisson\left(#1\right)}
\DeclareMathOperator{\Binomial}{Bin}
\newcommand{\Bin}[1]{\Binomial\left(#1\right)}
\newcommand{\N}{\mathbb{N}}
\newcommand{\expp}[1]{\exp \left(#1\right)} 
\newcommand{\exnd}[1]{\expp{-\D^{#1}}}
\newcommand{\expd}[1]{\expp{\D^{#1}}}
\newcommand{\prob}[1]{\mathbb{P}\left[#1\right]} 

\newcommand{\expec}[1]{\mathbb{E}\left[#1\right]} 
\newcommand{\roundup}[1]{\left\lceil#1\right\rceil} 
\newcommand{\set}[1]{\left\{#1 \right\}} 
\newcommand{\parens}[1]{\left( #1 \right)} 
\newcommand{\ceil}[1]{\left\lceil #1 \right\rceil} 
\newcommand{\card}[1]{\left|{#1}\right|} 

\newcommand{\whp}{whp}
\newcommand{\im}{{i-1}}
\newcommand{\parent}{\operatorname{parent}}

\newcommand{\bigo}[1]{O\mathopen{}\left(#1\right)\mathclose{}}
\newcommand{\smallomega}[1]{\omega\mathopen{}\left(#1\right)\mathclose{}}

\newcommand{\failedV}[2]{ 
\ifstrequal{#2}{}
{F}
{\ifstrequal{#1}{}
{F\mathopen{}\left(#2\right)\mathclose{}}
{F^{(#1)}\mathopen{}\left(#2\right)\mathclose{}}}
}

\renewcommand{\failedV}{F} 

\newcommand{\badV}[1]{  
\ifstrequal{#1}{}
{B}
{B\mathopen{}\left(#1\right)\mathclose{}}
}
\newcommand{\usefulE}[1]{S_{#1}} 
\newcommand{\gproperty}[0]{\mathcal{P}} 

\newcommand{\embedding}[1]{  
\ifstrequal{#1}{}
{\Phi}
{\Phi_{#1}}
}

\newcommand{\D}{\Delta}
\newcommand{\replacementgadget}[3]{ 
\ifstrequal{#1}{}
{\Gamma(#2,#3)}
{\Gamma^{(#1)}(#2,#3)}
} 

\newcommand{\length}{coverage\xspace}
\newcommand{\CC}{\ensuremath{\mathcal{C}}}
\newcommand{\tauout}{\tau_\textsc{Out}}

\renewcommand{\rq}{''}
\newcommand{\poly}{\operatorname{poly}}

\title{Optimally building spanning graphs in semirandom graph processes}
\author[Anastos]{Michael Anastos$^1$}
\author[Collares]{Maurício Collares$^2$}
\author[Erde]{Joshua Erde$^3$}
\author[Kang]{Mihyun Kang$^4$}
\author[Schmid]{Dominik Schmid$^4$}
\author[Sorkin]{Gregory B. Sorkin$^5$}

\address{$^1$Institute of Science and Technology Austria (ISTA), 3400 Klosterneuburg, Austria}
\email{michael.anastos@ist.ac.at}

\address{$^2$Instituto de Matem\'atica e Estat\'istica, Universidade de
S\~ao Paulo, Rua do Mat\~ao 1010, 05508-090 S\~ao Paulo, Brazil}
\email{collares@ime.usp.br}

\address{$^3$School of Mathematics, University of Birmingham, Birmingham, B15 2TT, UK}
\email{j.erde@bham.ac.uk}

\address{$^4$Institute of Discrete Mathematics, Graz University of Technology, Steyrergasse 30, 8010 Graz, Austria}
\email{\{kang,schmid\}@math.tugraz.at}

\address{$^5$Department of Mathematics, The London School of Economics and Political Science, London,
England}
\email{g.b.sorkin@lse.ac.uk}

\begin{document}
\newcounter{assertno}[theorem]
\setcounter{assertno}{0}

\newenvironment{assert}{\refstepcounter{assertno}\begin{list}%
    {\textbf{Assertion \theassertno: }\hfill}{\setlength{\topsep}{3pt}%
      \setlength{\partopsep}{0pt}\setlength{\parsep}{3pt plus 1pt}%
      \setlength{\leftmargin}{0.01mm}\setlength{\labelwidth}{0.01mm}%
      \setlength{\labelsep}{0pt}\renewcommand{\baselinestretch}{1.5}}%
  \item}{\end{list}}

\newenvironment{reason}{\begin{list}%
    {\textit{Reason: }\hfill}{\setlength{\topsep}{3pt}%
      \setlength{\partopsep}{0pt}\setlength{\parsep}{3pt plus 1pt}%
      \setlength{\leftmargin}{0.01mm}\setlength{\labelwidth}{0.01mm}%
      \setlength{\labelsep}{0pt}\renewcommand{\baselinestretch}{1.5}}%
  \item}{\end{list}}

\begin{abstract}
The semirandom graph process constructs a graph $G$ in a series of rounds, starting with the empty graph on $n$ vertices.
In each round, a player is offered a vertex $v$ chosen uniformly at random, and chooses an edge on $v$ to add to $G$. 
The player's aim is to make $G$ satisfy some property as quickly as possible. 
Our interest is in the property that $G$ contain a given $n$-vertex graph $H$ with maximum degree~$\D$.
In 2020, Ben-Eliezer, Gishboliner, Hefetz and Krivelevich showed that there is a semirandom strategy that achieves this, 
with probability tending to one as $n$ tends to infinity, in $\parens{1 + o_\D(1)}\frac{3 \D n}{2}$ rounds, where $o_\D(1)$ is a function that tends to $0$ as $\D$ tends to infinity. 
We improve this to $\parens{1 + o_\D(1)}\frac{ \D n}{2}$, which can be seen to be asymptotically optimal in $\D$. 
We show the same result for a variant of the semirandom graph process, namely the semirandom tree process introduced by Burova and Lichev, where in each round the player is offered the edge set of a uniformly chosen tree on $n$ vertices, and chooses one edge to keep.
\end{abstract}
\maketitle

\section{Introduction}\label{sec:introduction}
The \emph{random graph process} was introduced by Erd\H{o}s and R\'{e}nyi in their groundbreaking papers on random graphs \cite{ER59,ER60}. In this process, starting with the empty graph with vertex set $[n] \coloneqq \set{1,\ldots,n}$, a sequence of edges $e_1,e_2,\ldots, e_{\binom{n}{2}}$ is chosen, where in the $t$th step the edge $e_t$ is chosen uniformly at random from the set $\binom{[n]}{2} \setminus \{e_1,\ldots, e_{t-1}\}$
of edges which have not already been chosen. One of the motivations for studying this process is that it gives a natural coupling for the set of \emph{uniform random graphs} $G(n,m)$, where $G(n,m)$ is a graph chosen uniformly at random from all graphs with vertex set $[n]$ and with $m$ edges. Indeed, it is easy to see that for each $0 \leq m \leq \binom{n}{2}$ the graph $G_m \colon = ([n], \{e_1,\ldots, e_m\})$ has the same distribution as $G(n,m)$, and hence the random graph process leads to a natural coupling  of the random graphs $G(n,0),G(n,1), \ldots, G\left(n, \binom{n}{2}\right)$ such that
\[
G(n,0) \subseteq G(n,1) \subseteq \ldots \subseteq G\big(n, {\textstyle \binom{n}{2}}\big).
\]
This allows us to think of the structure of the random graph $G(n,m)$ as \emph{evolving} as we increase the parameter $m$. The random graph process has been extensively studied, and there has been much research into how this process is affected by restricting or changing the distribution of the edges chosen in each step (see monographs on random graphs~\cite{Bollobas2001,FrKa2016,JaLuRu00}).

For example, given some decreasing graph property $\mathcal{P}$, we could restrict our choice of edges by only allowing an edge $e_t$ in the $t$th step if the resulting graph $G_t$ has the property $\mathcal{P}$ (see e.g., \cite{ESW95}). In this way, the process would stop not with a complete graph, but with a maximal graph with the property $\mathcal{P}$. This gives a simple and explicit way to sample a random extremal graph with property $\mathcal{P}$, although the distribution on the set of such graphs may not be uniform. These models have been useful for the study of $H$-free graphs \cite{BK10,BR00,OT01}, which has important applications in Ramsey theory \cite{B09Triangle,BK10,K95}, as well as in the study of random planar graphs \cite{GSST08,KaMi2023} and in other contexts \cite{BFL15,KKMS20,KKLS18,KSV09,KOT16}. 

Other results focus on how the process is affected if we introduce some bias, either random or via some strategic agency, to the choice of the edge in each step. In particular, perhaps motivated by a paradigm in computer science known as the \emph{power of two random choices} \cite{AzBrKaUp1999,MRS01}, Achlioptas suggested a variation of the random graph process, now known as the \emph{Achlioptas process} (see \cite{BoFr2001}). Starting with the empty graph on $n$ vertices, in each step \emph{two} (or several) edges are sampled uniformly at random and one of those is selected according to some rule and added to the graph. We are then interested in how the choice of this rule can affect the evolution of the random graph process. Particular focus has been given to how it can affect the \emph{phase transition} this model undergoes, where the size of the largest component transitions from logarithmic to linear, specifically, how this phase transition can be delayed or accelerated \cite{BoFr2001,BoKr2006,FlGaSo2005,JaSp2012,SpWo2007} and how ``smoothly'' it occurs \cite{AcDsSp2009,B09Connec,RiWa2012,RiWa2011}.

Similar questions about the effect of strategic choice on random processes were considered by Ben-Eliezer, Hefetz, Kronenberg, Parczyk, Shikhelman and Stojakovi\'c \cite{BeHeKrPaShSt2020}, who studied the so-called \emph{semirandom graph process} suggested initially by Michaeli.
Starting with the empty graph $G_0$ on the vertex set $[n]$, in each round $t \in \N$ a player
we refer to as ``Builder''
is \emph{offered} a vertex $v_t$ 
chosen from $[n]$ uniformly at random and independently of all previous rounds. Builder then \emph{chooses} an edge $e_t$ incident to $v_t$ and adds it to the graph $G_{t-1}$, i.e., $E(G_t) := E(G_{t-1}) \cup \{e_t\}$. Equivalently, Builder is offered the edge set of a random \emph{star} on vertex set $[n]$ centred at $v_t$, and chooses an edge $e_t$ from this star. Builder's goal is to build a graph satisfying a given graph property $\gproperty$ (typically monotone increasing) as quickly as possible. For convenience, we also allow Builder to claim no edge in a round, which clearly does not give Builder any strategic advantage. 
Although recent, the semirandom graph process has received extensive attention,
with results on graph properties including containment of a perfect matching~\cite{BeHeKrPaShSt2020,GaMaPr2022PM}, Hamiltonicity~\cite{FGMPS23,FrSo2022, GaKaMaPr2022, GaMaPr2022HC}, containment of bounded-size subgraphs~\cite{BeMarPrRu2024,BeHeKrPaShSt2020}, and the chromatic and independence numbers~\cite{GaKaPr2024}.

\smallskip

Before proceeding, let us mention that we write ``whp'' to mean ``with high probability\rq, that is, with probability tending to $1$ as $n$ tends to infinity. 
We use the usual Landau notation, such as $\omega(\ln n)$, 
also with the assumption that $n \to \infty$. 
By contrast, we write $o_\D(1)$ to denote an arbitrary function 
that approaches 0 as $\D \rightarrow \infty$.

\smallskip

In this paper we concern ourselves with the property of constructing a copy of a given $n$-vertex graph $H$ of maximum degree $\D$. Alon \cite[Question 6.2]{BeHeKrPaShSt2020} asked whether every graph of bounded maximum degree can be constructed in $O(n)$ rounds. This was resolved affirmatively by Ben-Eliezer, Gishboliner, Hefetz and Krivelevich~\cite{BeGiHeKr2020},
who showed the following. 
\begin{theorem}[{\cite[Theorem 1.3]{BeGiHeKr2020}}]\label{thm:Ben_Eliezer_result}
Let $H$ be an $n$-vertex graph with maximum degree $\D = \D(n)$. Then, in the semirandom graph process, Builder has a strategy guaranteeing that \whp{},
after
    \[\begin{array}{c@{\quad}l}
       \parens{1 + o_\D(1)}\frac{\D n}{2} \text{ rounds} &\text{if } \D = \smallomega{\ln n}, and
       \\[+1ex]
       \parens{1 + o_\D(1)}\frac{3 \D n}{2} \text{ rounds} & \text{otherwise,} 
    \end{array}\]
    the constructed graph will contain a copy of $H$.
\end{theorem}

Because $H$ may contain $\frac{\D n}{2}$ edges, Builder needs at least that many rounds.
Thus, \Cref{thm:Ben_Eliezer_result} is asymptotically optimal for $\D = \smallomega{\ln n}$.
On the other hand, for smaller $\D$, 
it is a factor of 3 away from the trivial lower bound. 

Ben-Eliezer, Gishboliner, Hefetz and Krivelevich asked if the first bound in \Cref{thm:Ben_Eliezer_result} can be shown to hold for all $\D$~\cite[Problem 1.4]{BeGiHeKr2020}.
For constant $\D$, \Cref{thm:Ben_Eliezer_result} implies that $O(n)$ rounds suffice, and to say more would mean characterising the leading constant.
(As noted earlier, there is ongoing exploration of the constants for $\D=1$, namely for perfect matchings, and for $\D=2$, namely for Hamilton cycles.)
So, retaining the factor $1+o_\D(1)$, the question posed is really about the behaviour for $\D=\smallomega 1$.

Our main result answers this question affirmatively: Builder can build any $n$-vertex graph with maximum degree $\D$ in time asymptotically matching the trivial lower bound.
\begin{theorem}\label{thm:main:star}
Let $H$ be an $n$-vertex graph with maximum degree $\D = \D(n)$. Then, in the semirandom graph process, Builder has a strategy guaranteeing that whp after
\begin{align}\label{eq:thm:main:star}
    \parens{1+ o_\D(1)}\frac{\D n}{2}
\end{align}
rounds, the constructed graph will contain a copy of $H$. 
Furthermore, the same statement holds after $\Poi{\parens{1+ o_\D(1)}\frac{\D n}{2}}$ rounds.
\end{theorem}
As we will show (\Cref{PoEnough}), the two statements in \Cref{thm:main:star} are equivalent.

As with \Cref{thm:Ben_Eliezer_result}, explicit quantitative bounds for fixed $\D$ can be extracted from our proofs, but they will be far from optimal.
As in \cite{BeGiHeKr2020}, without loss of generality we will assume in our proofs that $\D$ is at least a sufficiently large constant.

\bigskip

Let us explain the main ideas allowing us to improve upon the bounds given in \Cref{thm:Ben_Eliezer_result}.
Here, as for many properties in semirandom processes, it is easy to construct an ``almost solution'' efficiently, and the hard part lies in ``completing'' this solution.
Indeed, if we arbitrarily identify the vertex set of the target graph $H$ with the vertex set $[n]$ of Builder's graph, then in the first $(1+o_\D(1)) \frac{\D n}{2}$ rounds,
Builder can greedily obtain nearly all the required edges. 
Calling a vertex incident to any missing edge \emph{failed}, \Cref{lem:greedy_analysis} shows that with high probability, at most $\exnd{0.95} n$ vertices fail.
(And \Cref{greedyLimit} shows that no greedy strategy can increase that $0.95$ exponent above~1, 
decreasing this ``subexponential'' fraction to ``superexponential''.)

From this point on, we aim to \emph{repair} failed vertices.
More precisely, if $v$ is a failed vertex and there exists a vertex $v'$ such that $v'$ is adjacent to all vertices in $N_H(v)$ and $v$ is adjacent to all vertices in $N_H(v')$, then swapping the labels of $v$ and $v'$ reduces the number of failed vertices.
In the time available, the probability of constructing any particular such substitution is constant (bounded away from~1) and thus to go beyond what can be done greedily requires a large set of \emph{swapping candidates} $C(v)$ dedicated to a single failed vertex~$v$. 
Intuitively (but read on), in the time remaining we cannot expect a vertex in $N_H(v)$ to be offered enough times to form edges to every swapping candidate $v'$,
so instead it is natural to use offers of $v'$ to form the edges to $N_H(v)$. 
For there to be a reasonable chance that $v'$ is offered $|N_H(v)|$ times requires spending $\D n$ rounds. 
At a high level, this is the strategy employed in \Cref{thm:Ben_Eliezer_result},
and the source of its extraneous additive $\D n$.

To match the lower bound, after the greedy phase Builder can only afford to play for another $\D n \cdot o_\D(1)$ rounds.
In this case, a swapping candidate will succeed with low probability, so 
the strategy above will work only if we can
dedicate superexponentially many swapping candidates to each failed vertex. 
This is possible if 
there are fewer than, say, $\exnd{1.04} n$ failed vertices,
but as said, \Cref{greedyLimit} shows that no greedy algorithm can produce that outcome.

Enabled by a technical improvement in the greedy phase (see \Cref{lem:extragreedy_analysis}), our main contribution is a phase (see \Cref{sec:bridging_phase}) that uses $\D n \cdot o_\D(1)$ rounds and repairs all but a superexponentially small number of failed vertices
(leaving at most $\exnd{1.04} n$ failed vertices, per \Cref{lem:bridging_analysis}).
We call it a \emph{bridging phase}, since it bridges the gap between the greedy phase and a regime where the swapping strategy above works in $\D n \cdot o_\D(1)$ rounds.
To put it another way, with the number of failed vertices bounded by $\exnd{c} n$,
it brings the exponent from  $c<1$ to $c>1$.

Our bridging phase consists of two subphases. 
In the first subphase --- contrary to the ``intuitively'' above --- for each failed vertex $v$, 
we claim edges $wv'$ between neighbours $w \in N_H(v)$ and candidates $v' \in C(v)$,
in rounds in which $w$ is offered, rather than $v'$.
This works because here we ensure (by excluding a tiny number of failed vertices, ``hubs'' and hub neighbours, per \Cref{def_hub}) that not too many failed vertices have $w$ as a neighbour,
and we deliberately choose $C(v)$ to be small, 
with the result that $w$ does not have too many required edges,
so it is likely to be offered enough times to claim all of them.
Thus, \emph{almost every} neighbour $w \in N_H(v)$ of $v$ will be joined to every $v' \in C(v)$ after this first subphase.

In the second subphase, we return to the original idea of claiming such edges $wv'$ in rounds in which $v'$ is offered. 
The key difference is that now a typical $v' \in C(v)$ is already connected to almost every $w \in N_H(v)$ by the first subphase, so it is likely that in the second subphase $v'$ is offered enough times to claim the remaining edges.
It is then very likely that \emph{some} $v' \in C(v)$ can successfully be swapped with~$v$.
(Edges between $v$ and $N_H(v')$ are claimed in rounds where the corresponding vertex in $N_H(v')$ is offered, and are not problematic.)
Only a superexponentially small number of failures will remain after this subphase
($\exnd{1.04} n$ failed vertices, per \Cref{lem:bridging_analysis}), and 
the original swapping strategy is then efficient enough to finish the construction.

\medskip

This result was discovered while studying the same problem for a modified building process, the semirandom \textbf{tree} process introduced by Burova and Lichev in 2023 \cite{BuLi2023}.
In this process, in the $t$th round Builder is offered the edge set of a spanning tree $T_t$ chosen uniformly at random from all spanning trees on the vertex set $[n]$ and independently of the previous rounds.
To maintain a clear distinction, we will refer to the usual semirandom graph process as the semirandom \textbf{star} process for the rest of this paper.
Burova and Lichev \cite{BuLi2023} showed, amongst other results, that in the semirandom tree process, perfect matchings and Hamiltonian cycles can be built in time asymptotically equal to their size.
(By contrast, strictly larger time is required in the semirandom star model.)
However, for clique factors, they could not establish even that time $O(n)$ suffices, remarking that they were unable to extend the methods of~\cite{BeGiHeKr2020} and asking if the analogue of \Cref{thm:main:star} might fail to hold in this model \cite[Question 1.5]{BuLi2023}.
We obtain the following theorem, analogous to \Cref{thm:main:star}, in particular answering the question of Burova and Lichev in the negative.

\begin{theorem}\label{thm:main:tree}
Let $H$ be an $n$-vertex graph with maximum degree $\D = \D(n)$. Then, in the semirandom \textbf{tree} process, Builder has a strategy guaranteeing that whp after
  \begin{align}
  \parens{1+ o_\D(1)}\frac{\D n}{2}
  \end{align}
  rounds, the constructed graph will contain a copy of $H$. Furthermore, the same statement holds after $\Poi{\parens{1+ o_\D(1)}\frac{\D n}{2}}$ rounds.
\end{theorem}
As with \Cref{thm:main:star}, the two statements in \Cref{thm:main:tree} are equivalent (see \Cref{PoEnough}).

The rest of this paper is structured as follows. In \Cref{sec:Preliminaries} we introduce all necessary notation and probabilistic tools. In \Cref{sec:greedy,sec:substitutions,sec:bridging_phase,sec:final_phase} we describe Builder's strategy in detail. In \Cref{sec:proof_main} we analyse the strategy to prove  \Cref{thm:main:star}. In  \Cref{sec:proof_main_tree} we detail the adaptations necessary to show \Cref{thm:main:tree} using a similar strategy. Finally, in \Cref{sec:Discussion} we discuss some open questions and avenues for future research.

\section{Preliminaries}\label{sec:Preliminaries}
In this section we state the graph theoretical and probabilistic results that we will use in our argument. 

\subsection{Hajnal-Szemer\'edi Theorem}

We will make use of the Hajnal-Szemer\'edi Theorem on equitable colourings ({\cite[Theorem 1]{HaSz1970}}). 
\begin{theorem}\label{thm:Hajnal-Szemeredi}
    Let $H$ be a graph with maximum degree $\D$. For every $i \in \N$ and every $k \geq \D^i +1$, there exists a $k$-partition of $V(H)$ such that 
    \begin{enumerate}[(i)]
        \item vertices in the same part are at pairwise distance at least $i+1$
        \item the part of largest cardinality contains at  most one more vertex than the part of smallest cardinality. 
    \end{enumerate}
\end{theorem}
Note that Hajnal and Szemer\'edi originally stated this theorem only for the case $i=1$.  However, for larger $i$ the result is implied by considering the auxiliary graph $H^i \supseteq H$ that is obtained by joining all vertices of $H$ that are at pairwise distance at most $i$ in $H$.
 
\subsection{Chernoff bounds}
We use the following Chernoff bound for Poisson random variables (see, e.g.,~\cite[Theorem 5.4]{MiUp17}) and binomial random variables (see, e.g.,~\cite[Theorem 2.1]{JaLuRu00}).
\begin{lemma}\label{lem:chernoff}
  Let $X$ be a random variable with mean $\mu > 0$ and either Poisson or binomial distribution. Then, for $t \geq 0$,
  \[
    \prob{X \geq \mu + t} \leq e^{-\mu} \left(\frac{e\mu}{\mu+t}\right)^{\mu+t} 
    \leq \left(\frac{e \mu}{t}\right)^t ,
  \]
  the weaker last form a convenience here.
  Moreover, for $0 \leq t \leq \mu$, it also holds that
  \[
    \prob{X \leq \mu - t} \leq \expp{-\frac{t^2}{2\mu}}\qquad\text{and}\qquad\prob{X \geq \mu + t} \leq \expp{-\frac{t^2}{3\mu}}.
  \]
\end{lemma}

\subsection{Concentration of a number of events}
We will repeatedly employ the following lemma, that if events for vertices some distance apart are independent, then the number of events is concentrated.

\begin{lemma} \label{lemma:concentration}
Let $d \in \mathbb{N}$ and let $0<p'<p<1$. Let $G$ be an $n$-vertex graph with
maximum degree at most $\D=\D(n)$. For every $x \in V(G)$, let $A_x$ be an
event which holds with probability at most $\exnd{p}$ and suppose that for any subset $X \subseteq V(G)$ of vertices at pairwise distance at least $d$, the events $\{A_x \colon x \in X\}$ are mutually independent. If $\D$ is sufficiently large,
then whp the number of $A_x$ which hold is at most $\exnd{p'}n$.
\end{lemma}

It is clear from Markov's inequality (and detailed below) that the lemma's conclusion holds with probability tending to 1 as $\D \to \infty$, but the lemma's ``whp'' means with probability tending to 1 as $n \to \infty$.

\begin{proof}
Let us define $q = \frac{1}{2p}$ and let $X := |\{x \in V(G) \colon A_x \text{ holds} \}|$. We consider two cases, $\D \geq \ln \ln n$ and $\D \leq \ln \ln n$.

In the first case, since $\expec{X} \leq \exnd{p}n$, by Markov's inequality,
\begin{align*}
\prob{X > \exnd{p'} n}\leq \frac{\exnd{p}n}{\exnd{p'}n} 
&=
o_\D(1) = o(1),
\end{align*}
using that (within this case) as $n \to \infty$, necessarily $\D \to \infty$.

In the second case, by the Hajnal-Szemer\'edi Theorem (\Cref{thm:Hajnal-Szemeredi}) we may partition $V(G)$ into $t:=\D^{d-1}+1$ parts $V_1,\ldots, V_t$ such that for each $i \in [t]$ we have $n_i := |V_i|  \geq \frac{n}{\D^{d-1}+1} -1 \geq \frac{n}{2\D^{d-1}}$ for sufficiently large $n$ and vertices in the same part are at pairwise distance at least $d$. In particular, the events $\{A_x \colon x \in V_i\}$ are mutually independent for each $i$.

Hence, if we let $X_i := |\{x \in V_i \colon A_x \text{ holds} \}|$, then each
$X_i$ is binomially distributed, with $\expec{X_i} \leq m_i \coloneqq \exnd{p} n_i$ and $X = \sum_{i=1}^t X_i$. Note that our choice of $q$ implies that $m$ is almost linear in $n$. Indeed, since $\D^{d-1}$ and $\D^p$ are both $o(\ln n)$,
\begin{align*}
 \ln m_i &\geq -\D^p + \ln n - \ln(2\D^{d-1}) \geq (1-o(1)) \ln n .
\end{align*}
Then, by the Chernoff bound (\Cref{lem:chernoff}),
\begin{align*}
\prob{X_i \geq 10 \, m_i}
  &\leq \parens{\frac{e \cdot m_i}{10 \, m_i}}^{10 m_i} 
  < e^{-10 m_i}
  \leq \expp{-10 \: (1-o(1)) \: \ln n}
  < n^{-9}.
\end{align*}
Hence, since $t = \D^{d-1} + 1$ and $\D \leq \ln \ln n$, by a union bound the probability that there is some $i \in [t]$ such that $X_i \geq 10 m_i$ is at most $n^{-8} = o(1)$. Hence, with probability $1 - o(1)$,  
$$X = \sum_{i=1}^t X_i \leq 10 \sum_{i=1}^t m_i = 10 \exnd{p} \sum_{i=1}^t n_i
= 10 \exnd{p}n \leq \exnd{p'} n ,$$
where the last inequality uses that $\D$ is sufficiently large.
\end{proof}

\subsection{Poissonisation}\label{s:Poissonisation}
A useful technique that we will employ is \emph{Poissonisation} (see, e.g.,~\cite[Section 5.4]{MiUp17}).
As in the second part of \Cref{thm:main:star}, we run the semirandom process not for a fixed number of rounds $\tau n$ but a Poisson-distributed number of rounds $\Poi{\tau n}$.
Equivalently, we may think of running the semirandom process for a fixed time $\tau n$, with offers occurring at times given by a Poisson process of rate~$1$.
(In either case, each offer, independently, is of a uniformly random vertex.)

Analysing the Poissonised process is simpler because,
by the subdivision property of the Poisson process, the process described is equivalent to each vertex being offered at times given by a Poisson process of rate $1/n$, the processes all independent.
Hence, the following holds:
\begin{equation}\label{e:Poissonisation}
  \begin{split} &\text{In $\Poi{\tau n}$ rounds, each vertex is offered $\Poi{\tau}$ times, and the number of offers}\\
  &\qquad\qquad\qquad\,\,\,\text{for different vertices are mutually independent.}
  \end{split}
\end{equation}

\begin{remark} \label{PoEnough}
The two statements in \Cref{thm:main:star,thm:main:tree} are equivalent.
\end{remark}

\begin{proof}
For either theorem,
the Poisson case would establish that, for some strategy and some positive $f(\D) = o_\D(1)$,
taking $\mu = (1+f(\D)) \frac{\D n}{2}$, 
with $\Poi{\mu}$ offers the strategy succeeds whp.
By the Chernoff bound (\Cref{lem:chernoff}), this number of offers $\Poi\mu$ is whp at most 
$(1+\D^{-1/3}) \mu$:
the failure probability is
\begin{align*} 
\prob{ \Poi{\mu} > (1+\D^{-1/3}) \mu}
 &\leq \expp{- \frac{(\D^{-1/3} \mu)^2}{3 \mu}} 
 \leq \expp{-\frac{\D^{1/3} n}{6}} 
 =o(1). 
\end{align*}
This implies that the strategy succeeds whp in a fixed $(1+\D^{-1/3}) \mu = (1+o_\D(1))\frac{\D n}{2}$ rounds.

A similar argument shows the converse.
\end{proof}

\subsection{Balanced orientations}

We will also use a lemma, shown in \cite[Lemma 2.4]{BeGiHeKr2020}, on ``balanced'' orientations of a graph.

\begin{lemma}[\cite{BeGiHeKr2020}]\label{lem:graph_orientation}
    Let $H$ be a graph. Then there exists an orientation $\vec{H}$ of the edges of $H$ such that $d^+_{\vec{H}}(v) \leq \roundup{d(v)/2}$ for all $v \in V(H)$.
\end{lemma}
\Cref{lem:graph_orientation} can be proved by adding an extra vertex adjacent to vertices of odd degree, finding an Eulerian circuit in each connected component of the resulting graph, and deleting the extra vertex.

\section{The Greedy phase}\label{sec:greedy}

By \Cref{PoEnough}, it suffices to prove the Poisson case of \Cref{thm:main:star}.
In the following sections we formally describe and analyse the different phases of Builder's strategy. In each phase we assume Builder is given a (potentially empty) graph $G$ with certain properties and show that after a given number of rounds Builder's graph will have a new set of properties. In \Cref{sec:proof_main} we combine these phases to obtain a strategy satisfying \Cref{thm:main:star}. This way, each section is self-contained, with no need to track the time used in other phases. Throughout the remaining sections, let $H$ be an arbitrary graph on $n$ vertices with maximum degree $\D$. 

The first phase will be a \textsc{Greedy} {phase}. In this phase,  it will be notationally convenient to assume that $V(H) = V(G) = [n]$ by choosing an arbitrary embedding of $H$. We will fix an orientation $\vec{H}$ of $H$ satisfying the conclusion of \Cref{lem:graph_orientation}.
In this way, we will refer to edges incident to $v$ in $H$ as \emph{inedges} if they are oriented towards $v$ in $\vec{H}$, or \emph{outedges} otherwise. 

There will be two subphases of the \textsc{Greedy} phase, one for claiming outedges and one for claiming inedges. 

\begin{definition} \label{def:InOut}
When following the \textsc{Out} (resp.\ \textsc{In}) rule, if Builder is offered a vertex $v$, then Builder claims any of $v$'s unclaimed outedges (resp.\ inedges).
\end{definition}

\begin{remark} \label{rem:largeD}
If $\D = \smallomega{\ln n}$, the \textsc{Out} rule alone suffices to construct a copy of $H$ in the desired time and prove \Cref{thm:main:star}. 
\end{remark}
\begin{proof}
Run the \textsc{Out} rule for $\Poi{\tauout n}$ rounds with 
$\tauout = \roundup{\D/2} + 2\sqrt{\D \ln n}$;
observe that $\tauout = (1+o_\D(1)) \D/2$.
This strategy succeeds if every vertex is offered at least $\roundup{\D/2}$ times.
The probability that a given vertex is not offered this many times is, 
by \eqref{e:Poissonisation} and the Chernoff bound (\Cref{lem:chernoff}), at most
\begin{align*}
\prob{\Poi{\tauout} < \roundup{\frac{\D}{2}}}
 &\leq \expp{-\frac{4\D \ln n}{2(1+o_\D(1)) \D}}
 \leq \expp{-1.9 \ln n} 
 = n^{-1.9}.
\end{align*}
By the union bound, the probability of any failure is $o(1)$.
\end{proof}

\begin{remark} \label{assumeDsmall}
In proving \Cref{thm:main:star}, we may henceforth assume that $\D \leq \ln n \ln \ln n$,
since larger values are addressed by \Cref{rem:largeD}. 
\end{remark}

As said, we will use two subphases. 

\subsection{First subphase: OutGreedy}
In this subphase, recalling \Cref{def:InOut}, Builder will follow the \textsc{OutGreedy} strategy defined below. 

\begin{strategy}[\textsc{OutGreedy}]\label{strategy:greedy}
    The \textsc{OutGreedy} strategy consists of following the \textsc{Out} rule 
    for $\Poi{\tau_{\textsc{Out}} \cdot n}$ rounds, where $\tau_{\textsc{Out}} = \roundup{\D/2} + \D^{0.99}$.
\end{strategy}

Thus, Builder spends most of the allotted time claiming outedges. Although this is a very simple strategy, Builder can ``almost'' construct a copy of $H$ using only this strategy. To make this more precise, call an edge $e$ of $H$ \emph{failed} if it is not in $G$, and a vertex $v$ a \emph{failure} or \emph{failed vertex} if it is incident (in $H$) to at least one failed edge. 
We will say a vertex or edge is \emph{successful} if it is not failed. 

\begin{lemma}\label{lem:greedy_analysis}
Let $G$ be a graph. If Builder follows \Cref{strategy:greedy}, then \whp{} afterwards there are at most $\exnd{{0.95}}n$ failed vertices in the resulting graph $G'$.  
\end{lemma}

\begin{proof}
Let us call a vertex an \emph{outfailure} if it has outedges not contained in $G'$ and denote by $F^+$ the set of outfailures. Clearly each outfailure was offered fewer than $\roundup{{d(v)}/{2}}$ times during this subphase. On the other hand, by \eqref{e:Poissonisation} each vertex is offered $\Poi{\tau_{\textsc{Out}}}$ times in this subphase. Hence, by the Chernoff bound (\Cref{lem:chernoff}), for any vertex $v \in [n]$,
\begin{align}\label{eq:greedy_outfailure_prob}
    \prob{v \text{ is an outfailure}} &\leq \prob{\Poi{\tauout} < \roundup{\frac{\D}{2}}}\leq \expp{- \frac{\parens{\tauout-\roundup{\frac{\D}{2}}}^2}{2 \tauout}} \notag \\
    &
    \leq \expp{- \frac{\D^{1.98}}{6 \D}} 
    \leq \exnd{{0.97}}, 
\end{align} where the last inequality uses that $\D$ is sufficiently large. 

Note that if none of $\set{v} \cup N_H(v)$ are outfailures, then all inedges and outedges of $v$ are present, i.e., $v$ is not a failure. Hence, by a union bound,
\begin{align}\label{eq:greedy_failure_prob}
    \prob{v\text{ is a failure}} \leq \prob{\left(\set{v} \cup N_H(v)\right) \cap F^+ \neq \emptyset} \leq (\D+1) \exnd{{0.97}} \leq \exnd{{0.96}}.
\end{align}

Finally, by \eqref{e:Poissonisation} for a set of vertices at pairwise distance at least 3, failures are mutually independent. It follows from \Cref{lemma:concentration} that whp the total number of failures is at most $\exnd{0.95} n$.
\end{proof}

\subsection{Second subphase: InGreedy}

For reasons explained in \Cref{subsec:batches_limitations}, it will be useful to control the number of vertices which are adjacent to many failures. Specifically, we introduce the following definition.
\begin{definition} \label{def_hub}
A \emph{failure hub}, or \emph{hub} for short, is a vertex with at least $\D^{0.1}$ failed neighbours. Note that the vertex need not be a failure itself.
\end{definition}

In what follows, we will require that the number
of hubs is extremely small, superexponentially so in $\D$. This would follow relatively quickly from \eqref{eq:greedy_failure_prob} and \Cref{lemma:concentration} if the failure of neighbours of a fixed vertex happened somewhat `independently'. However, after the first subphase there are quite large dependencies between vertex failures in a fixed neigbhourhood. Indeed, if a vertex $v$ is not offered during this phase, and so Builder does not claim any of its outedges, then it and all of its neighbours are failures.

To this end, we follow the \textsc{OutGreedy} subphase by this \textsc{InGreedy} subphase in which we claim \emph{inedges}.

\begin{strategy}[\textsc{InGreedy}]\label{strategy:extragreedy}
Starting with an arbitrary graph, the \textsc{InGreedy} strategy consists of following the \textsc{In} rule for $\Poi{\tau_{\textsc{In}} \cdot n}$ rounds, where $\tau_{\textsc{In}} = \D^{0.99}$.
\end{strategy}

We now analyse the combined effect of the two subphases.

\begin{lemma}\label{lem:extragreedy_analysis}
Let $G$ be a graph. If Builder follows \Cref{strategy:greedy} and then subsequently \Cref{strategy:extragreedy}, then \whp{} afterwards there are at most $\exnd{{1.05}}n$ hubs in the resulting graph.
\end{lemma}

\begin{proof}
As in \Cref{lem:greedy_analysis}, let us write $G'$ for Builder's graph after the \textsc{OutGreedy} strategy, and $F^+$ for its set of outfailures. 
Denote by $G''$ Builder's graph after the \textsc{InGreedy} strategy.

For each $v \in [n]$, let $\mathcal{E}(v)$ denote the event $\{|N^2[v] \cap F^+| \leq f\}$, with $f := \D^{0.1}/2$ and where $N^2[v] = \{v\} \cup N(v) \cup N(N(v))$. Note that $|N^2[v]| \leq \D^2+1$. 
Informally, $\mathcal{E}(v)$ is the event that there are few outfailures near~$v$, and we will show that it is likely to hold.
If it does hold, then of course within $N(v)$ there are at most $f$ outfailures.
Also, while a $w \in N(v)$ may fail due to infailures, if $\mathcal{E}(v)$ holds then $w$ has at most $f$ infailed edges, and if $w$ receives $f$ or more offers, it can repair all of these. 
Thus, if $\mathcal{E}(v)$ holds and at most $f$ neighbours of $v$ fail to get $f$ offers, 
then $v$ has at most $2f$ failed neighbours and is not a hub. 
That is, $v$ is a hub only in the event $\overline{\mathcal{E}(v)}$ or if more than $f$ neighbours of $v$ fail to get $f$ offers.

Since by \eqref{e:Poissonisation} the events $\{w \in F^+ \colon w \in N^2[v]\}$ are independent (all outfailures are independent), we may use the Chernoff bound (\Cref{lem:chernoff}) together with \eqref{eq:greedy_outfailure_prob} to bound
  \begin{align}\label{eq:greedy_many_failures}
      \prob{\overline{\mathcal{E}(v)}} 
      \leq \left(\frac{e \cdot (\D^2 + 1) \exnd{{0.97}}}{\D^{0.1}/2}\right)^{\D^{0.1}/2} 
      \leq \left(\exnd{{0.965}}\right)^{\D^{0.1}/2} 
      \leq \exnd{1.06}.
  \end{align}

By \eqref{e:Poissonisation} and the Chernoff bound (\Cref{lem:chernoff}), the probability that a vertex is offered fewer than ${f = \D^{0.1}/2}$ times in the \textsc{InGreedy} subphase is at most
  \begin{equation}\label{eq:poisson_lower_tail_app}
    \prob{\Poi{\tau_{\textsc{In}}} < \D^{0.1}/2} \leq \expp{-\frac{(\D^{0.99} - \D^{0.1}/2)^2}{2\D^{0.99}}} \leq \exnd{{0.98}}.
  \end{equation}  
Again by \eqref{e:Poissonisation}, these events are independent for different vertices, so, using \eqref{eq:poisson_lower_tail_app}, the probability that more than $f$ vertices in $N(v)$ are offered fewer than $f$ times is at most
\begin{align}
\prob{\Bin{|N(v)|, \exnd{{0.98}}} > f}
  &\leq  \prob{\Bin{\D, \exnd{{0.98}}} > \D^{0.1}/2} \notag
  \\& \leq  \left(\frac{e \D \cdot \exnd{{0.98}}}{\D^{0.1}/2}\right)^{\D^{0.1}/2} \label{nohub2.1}
  \\& \leq \left(\exnd{{0.98}}\right)^{\D^{0.1}/2} 
  \leq \exnd{{1.07}} ,  \label{nohub2}
\end{align}
using in \eqref{nohub2.1} the Chernoff bound (\Cref{lem:chernoff}).

Summing \eqref{eq:greedy_many_failures} and \eqref{nohub2},
\begin{align} \label{eq:greedy_hub_prob:2}
   &\prob{v\text{ is a hub after the \textsc{Greedy} phase}} 
 \leq 2 \exnd{1.06} .
\end{align}

For a set of vertices at pairwise distance at least 5, by \eqref{e:Poissonisation} the events that these vertices are hubs in $G''$ are mutually independent. 
It follows from \Cref{lemma:concentration} that whp there are at most $\exnd{{1.05}}n$ hubs in $G''$. 
\end{proof}

\begin{remark} \label{whyIn}
If we used only the \textsc{Out} rule we could not hope to obtain the result of \Cref{lem:extragreedy_analysis}.
In the allotted time we could not extend the \textsc{OutGreedy} subphase significantly,
so it is inevitable that a vertex $v$ receiving too few offers would cause many neighbours to fail, making $v$ a hub.
That the \textsc{In} rule leads to a better result, as shown by the lemma, 
may intuitively be ascribed to overcoming the focussed nature of the semirandom star process:
neighbours of $v$ can independently repair their own failures, 
a process more robust than relying on $v$ alone. 
\end{remark}

\begin{remark}\label{greedyLimit}
Although the \textsc{Greedy} strategy is efficient to get from an empty graph to most of $H$, it is insufficient to complete a copy of $H$ in the desired time. 
For any constant $c$, in time $\Poi{ c n}$ (more than our total budget, if $c > \D/2$), each vertex receives $\Poi c$ offers. However, a given edge $vw$ cannot be constructed if both $v$ and $w$ receive no offers, and this occurs with probability $\approx \exp(-2c)$. 
While \eqref{eq:greedy_failure_prob} shows that the failure probability can be limited to $\exnd{0.96}$, the above shows that no greedy strategy can boost that $0.96$ exponent above~1
in $O(\D n)$ rounds.
\end{remark}

\section{Repairing failures}\label{sec:substitutions}

As just mentioned, if $\D = \bigo{\ln n}$, then the \textsc{Greedy} strategy has
limitations. Therefore, to reduce the number of failures further after the \textsc{Greedy} phase, we will adapt an idea introduced in \cite{BeGiHeKr2020}. Given a failed vertex $v$, our hope is to find a vertex $v'$ such that after swapping the labels of $v$ and $v'$ in $G$, neither vertex will be a failure. This will be the case precisely if Builder's graph contains all edges from $v$ to neighbours of $v'$ (in $H$) and all edges from $v'$ to neighbours of $v$ (in $H$). 
(See \Cref{fig:swaps}.)
Denoting this set of required edges by the \emph{gadget}
\[ R(v, v') := \{v'w \colon w \in N_H(v)\} \cup \{vw' \colon w' \in N_H(v')\}, \]
if there exist a failed vertex $v$ and a vertex $v'$ such that all edges in $R(v,v')$ are present in Builder's graph, then we say $v$ is \emph{repairable} (using $v'$). Note that if $v$ is repairable using $v'$ and we swap the labels of $v$ and $v'$, then neither vertex will be failed. We refer to this process as \emph{swapping} $v$ and $v'$.

For efficiency, we will attempt to repair a set $I$ of failed vertices simultaneously.
We will avoid trying to repair two adjacent vertices $v$ and $w$ at once, since repairing $v$ with $v'$ means that $v'$ has an edge to $w$, but repairing $w$ means replacing it with some $w'$, and there is no reason to expect the necessary edge $v'w'$ to be present. 
So we will restrict $I$ to be an independent set in $H$.
For the same reason, no replacement candidate $v'$ should be in $N_H(I)$.
To avoid confusion, and since anyway $I$ will be small, we will also avoid choosing any replacement candidate $v'$ from $I$.
We will also make the natural restriction of choosing distinct replacement candidates for each $v \in I$.
That is, for each $v \in I$ we define a set $C(v)$ of swapping candidates,
with these sets disjoint from $I \cup N_H(I)$ and from one another. 
As discussed in the introduction, with $v' \in C(v)$ and $w' \in N_H(v')$, we will try to make edges between $w'$ and $v$ using offers of $w'$.
So furthermore, to avoid ``overburdening'' $w'$, we will insist that the replacement candidates have pairwise disjoint \emph{neighbourhoods}. Finally, for convenience, we will also insist that all candidate sets $C(v)$ have the same cardinality.

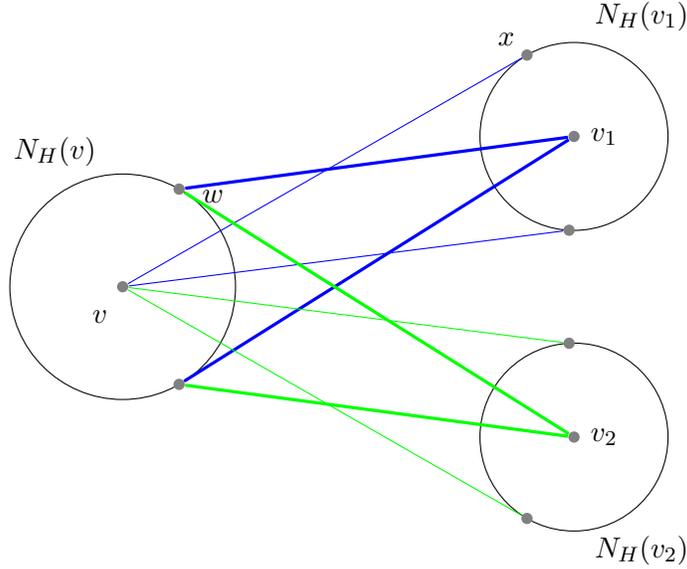
\begin{figure}[t]
        \centering
\begin{tikzpicture}[
    node distance=3cm,
    every node/.style={circle, draw, minimum size=0.15cm},
    connection/.style={thick},
    blue line/.style={connection, blue},
    green line/.style={connection, green!70!black}
]

\draw[draw=none] (0,0) circle (1.25cm);
\draw[draw=none] (4,2) circle (1cm);
\draw[draw=none] (4,-2) circle (1cm);

\draw (-2,0) circle (1.5cm);
\draw (4,2) circle (1.25cm);
\draw (4,-2) circle (1.25cm);

\node[circle, fill=gray, inner sep=1pt, draw=none] (v) at (-2,0) {};
\node[circle, fill=gray, inner sep=1pt, draw=none] (v1) at (4,2) {};
\node[circle, fill=gray, inner sep=1pt, draw=none] (v2) at (4,-2) {};

\node[circle, fill=gray, inner sep=1pt, draw=none] (w) at ($(-2,0) + (60:1.5)$) {};
\node[circle, fill=gray, inner sep=1pt, draw=none]  (w2) at ($(-2,0) + (-60:1.5)$) {};

\node[circle, fill=gray, inner sep=1pt, draw=none] (u1) at ($(4,2) + (120:1.25)$) {}; 
\node[circle, fill=gray, inner sep=1pt, draw=none]  (u2) at ($(4,2) + (-93:1.25)$) {};

\node[circle, fill=gray, inner sep=1pt, draw=none]  (t2) at ($(4,-2) + (93:1.25)$) {};
\node[circle, fill=gray, inner sep=1pt, draw=none]  (t1) at ($(4,-2) + (-120:1.25)$) {};

\node[draw=none] at (-2.3,-0.4) {$v$};
\node[draw=none] at (4.4,2) {$v_1$};
\node[draw=none] at (4.4,-2) {$v_2$};
\node[draw=none] at (-0.8,1.2) {$w$};
\node[draw=none] at (3.1,3.3) {$x$}; 

\node[draw=none] at (-2.9,1.8) {$N_H(v)$};
\node[draw=none] at (4.9,3.6) {$N_H(v_1)$};
\node[draw=none] at (4.9,-3.5) {$N_H(v_2)$};

\draw[blue] (v) -- (u1);
\draw[blue] (v) -- (u2);
\draw[blue,very thick] (w) -- (v1);
\draw[blue,very thick] (w2) -- (v1);
\draw[green] (v) -- (t1);
\draw[green] (v) -- (t2);
\draw[green,very thick] (w) -- (v2);
\draw[green,very thick] (w2) -- (v2);
\end{tikzpicture}
\caption{%
A failed vertex $v$, a swapping candidate $v_1$ and its gadget $R(v,v_1)$ (in blue), and a second swapping candidate $v_2$ and its gadget $R(v,v_2)$ (in green). 
Vertex $w$ is an illustrative neighbour of $v$; likewise $x$ of $v_1$. 
All vertices shown distinct are distinct, and 
the gadgets for $v$ are vertex-disjoint from those for 
other failed vertices considered at the same time
(except that $w$ may be a neighbour of other failed vertices).
If all blue edges $R(v,v_1)$ can be made, the $v_1$ gadget succeeds:
swapping the labels of $v$ and $v_1$ means both will have edges to their respective $H$-neighbourhoods.
Likewise for green and candidate $v_2$.
Thin edges will be easy to build. Construction of thick edges is discussed in the text; see particularly \Cref{subsec:batches_workaround}.}
\label{fig:swaps}
\end{figure}

To sum up, our constraints for swapping candidates are captured in the definition below. 
\begin{definition}\label{d:candidatelists}
    Let $I \subseteq V(H)$ be a non-empty independent set of failed vertices. A \emph{candidate list} $\CC$ for $I$ is a collection $(C(v))_{v \in I}$ of subsets of $V(H)$ of equal cardinality such that:
    \begin{enumerate}[$(i)$]
        \item\label{i:cand1} 
        The sets $(C(v))_{v \in I}$ are pairwise disjoint and disjoint from $I$.
        \item\label{i:cand2} 
        The set $I \cup \bigcup_{v \in I} C(v)$ has minimum distance at least~2.
        \item\label{i:list:dist_three} 
        The set $\bigcup_{v \in I} C(v)$ has minimum distance at least~3.
    \end{enumerate}
    The \emph{\length} of $\CC$ is the cardinality of $C(v)$ for any $v \in I$. An edge in the set $\bigcup_{v \in I}\bigcup_{v' \in C(v)} R(v, v')$ is called \emph{$\CC$-relevant}.
\end{definition}

In all cases, we will construct the candidate list $\CC$ greedily. 
That is, each vertex in $I$ ``blocks'' at most $\D+1$ vertices from consideration,
each candidate $v'$ blocks at most $\D^2+1$ vertices,
and any non-blocked vertex can be added to a smallest set $C(v)$.

In general, Builder's aim is to claim $\CC$-relevant edges whenever possible. However, to simplify analysis, we will generally define a set $A$ of \emph{active} vertices and only consider rounds in which Builder is offered a vertex of this set, where the definition of $A$ depends on the phase of the building process. 

In Sections~\ref{sec:bridging_phase}~and~\ref{sec:final_phase}, we will describe two phases in which we repair failures. The former is quite delicate, and will reduce the number of failed vertices to $\exnd{1+\varepsilon} n$, bypassing the limitations of the \textsc{Greedy} strategy. This allows us to be less careful in the latter phase since we can construct candidate lists with exponential \length, leading to a very high probability of success for swapping a vertex.

The strategy in both of these phases follows the same ``metarule''.

\begin{metarule}[The \textsc{Meta}$(\CC,A)$-rule]\label{metarule}
    Let $I \subseteq V(H)$ be an independent set of failed vertices, $\CC$ a candidate list for $I$, and $A \subseteq V(H)$ a set of \emph{active} vertices. When offered a vertex $v \in A$, Builder claims any unclaimed $\CC$-relevant edge incident to $v$, if such an edge exists. 
\end{metarule}

After a certain number of rounds of applying the MetaRule,
we \emph{execute} its results,
performing a maximum set of swaps consistent with $\CC$.

\begin{execution}
For every $v \in I$, if there is any $v' \in C(v)$ such that the edges of $R(v, v')$ are all present in Builder's graph, we arbitrarily pick any such $v'$ and swap the labels of $v$ and $v'$.
\end{execution}

\begin{remark}
After Execution, every $v \in V$ previously successful (not failed) remains so, 
and every $v \in I$ for which there is any $v' \in C(v)$ such that the MetaRule's application yielded all edges in $R(v,v')$, becomes successful.
\end{remark}
\begin{proof}
This follows from \Cref{d:candidatelists} properties \ref{i:cand1} and \ref{i:cand2};
also see \Cref{fig:swaps}.
No neighbour of any element in $I \cup \bigcup_{v \in I} C(v)$ is in that set, so
the neighbours can be thought of as a fixed substrate.
Therefore any previously successful vertex not swapped (relabeled) remains successful.

Furthermore, an executed swap of $v \in I$ and $v' \in C(v)$ will remain successful given subsequent swaps, say of $u$ and $u'$.
Indeed, because the neighbours remain fixed, this could fail only if $v$ or $v'$ were relabeled by a $u$--$u'$ swap, i.e., unless $v$ or $v'$ is equal to $u$ or $u'$. 
Since $u,v \in I$ and $u',v' \in \bigcup_{x \in I} C(x)$ and these two sets are disjoint by construction, that can only occur if $u=v$ (contradicting that we execute at most one swap per vertex in $I$) or if $u'=v'$ (in which case, since $u' \in C(u)$ and $v' \in C(v)$, we conclude that $u=v$ by property \ref{i:cand1} of \Cref{d:candidatelists}).
\end{proof}

\subsection{Batches and limitations}\label{subsec:batches_limitations}
Let us sketch a naive strategy to repair vertices, and give some crude estimates of the parameters involved and the limitations of this strategy. 

By the Hajnal-Szemer\'edi Theorem (\Cref{thm:Hajnal-Szemeredi}), we can divide the set $F$ of failed vertices into $\D+1$ independent sets, which we call \emph{batches}, of roughly equal size. Given a batch $I$, we can greedily construct a candidate list $\CC$ for $I$ from the vertices in $V(G) \setminus F$ satisfying \Cref{d:candidatelists}, 
with \length~$s$. Naively it would seem that to maximise the chances of success, we would like to maximise $s$.

Now, up to unimportant factors polynomial in $\D$, $s$ can be at most $n/\card{F}$. However,
\Cref{greedyLimit} means that the greedy algorithm must leave at least $\expp{-2\D}n$ failed edges, and (up to $\D$ factors) the same number of failed vertices, implying that in the best case $s \lessapprox \expp{2\D}$. Note that, in reality, \Cref{lem:greedy_analysis} would give us a slightly smaller $s \approx \expp{\D^{0.95}}$.

For $v \in I$, $v' \in C(v)$ and $w \in N_H(v)$, how will we claim $\CC$-relevant edges $(v',w)$
(the thick edges in \Cref{fig:swaps})?
For each such $w$, there are precisely $s$ such edges, one for each $v' \in C(v)$.
On the other hand, for each $v' \in C(v)$, $v'$ is a replacement candidate only for $v$,
and thus incident to at most $\D$ such edges. 
Since $s$ is much larger than $\D$, we will be much more successful claiming such edges on offers of $v'$.

Builder will spend $\Poi{\tau n}$ rounds per batch.
The chance that a given $v' \in C(v)$ receives $\D$ offers, to make edges to all of $N(v)$, is crudely estimated by $\prob{\Poi{\tau} \geq \D} \approx \parens{\frac{\tau}{\D}}^\D$. 
So, the probability that a given $v$ is repairable is, again estimating crudely with a union bound,
\begin{align}\label{eq:repair_prob_approx}
 \prob{v \text{ is repairable}} \lessapprox  |C(v)| \parens{\frac{\tau}{\D}}^\D \lessapprox  s \parens{\frac{\tau}{\D}}^\D.
\end{align}

However, after the Greedy stage we only have time for $o_\D(1)\D n$ further rounds, so in order to spend $\Poi{\tau n}$ for each of $\D$ batches requires that $\tau = o_\D(1) < 1$.
So, estimating crudely, the probability that a fixed vertex gets repaired is
\begin{align} \label{eq:repair_prob_approx2}
  \prob{v \text{ is repairable}} 
   \lessapprox s \parens{\frac{\tau}{\D}}^\D 
   \lessapprox e^{2\D} \D^{- \D  } = o_\D(1),
\end{align} 
and only a tiny proportion of the vertices would be repaired by this strategy.

\subsection{Workaround} \label{subsec:batches_workaround}
To make this work, we employ a ``bridging'' phase in our strategy, which uses two key ideas. First, by excluding hubs, we are able to divide $F$ into fewer independent batches and hence spend more time on each batch (increasing $\tau$). 

Second, and more importantly, we make the unintuitive choice to artificially reduce the coverage $s$. 
In itself, this would make each vertex less likely to be swapped.
However, consider a thick edge in \Cref{fig:swaps}: a $\CC$-relevant edge $(v',w)$ with $v \in I$ and $w \in N_H(v)$.
With $s \leq \D$ it becomes more efficient to claim such edges with offers of $w$.
Doing so, we will be able to claim most of the $\CC$-relevant edges incident to most $v' \in C(v)$.
At this point, a typical $v' \in C(v)$ only needs to be offered a small number of times to claim the remaining $\CC$-relevant edges $(v',w)$, and so we switch and claim such edges with offers of $v'$.

We will show that the following the above strategy results in just $\exnd{{1.04}} n$ failed vertices. 

Then, in the ``final'' phase, the method works as presented in \Cref{subsec:batches_limitations}: we can take $s$ about $\expp{\D^{1.04}}$ (and larger in subsequent iterations), 
at which point the quantity $s(\frac 1 \D)^\D$ in \eqref{eq:repair_prob_approx} becomes large, and repair becomes overwhelmingly likely.
That is, in \Cref{fig:swaps}, thick edges like $(w,v_1)$ are made on offer of $v_1$: it is unlikely that $v_1$ receives $\D$ offers, but with $s$ superexponential, it is likely that some $v_i$ succeeds.

\section{The bridging phase}\label{sec:bridging_phase}

Let $G$ be a graph for which the set of failed vertices, denoted by $F$, has size at most $\exnd{{0.95}}n$, and the set of hubs, denoted by $M$, has size at most $\exnd{{1.05}}n$. Our goal in the \textsc{Bridging} phase is to reduce the number of failed vertices to $\exnd{{1.04}}n$.

For the analysis, it will be crucial that we do not attempt to repair failed vertices which are or are adjacent to a hub. However, this will not affect our goal, since there are at most $(\D+1) \exnd{{1.05}}n \leq \exnd{{1.04}}n/2$ such vertices. So, let us denote by
\begin{align} \label{non-hub-adj}
F' &= F \setminus (M \cup N_H(M))
\end{align}
this subset of failed vertices. 
Throughout this section, take 
\begin{align*}
  k = \D^{0.1}.
\end{align*}

Because $F'$ has no hubs, by the Hajnal-Szemer\'edi equitable colouring theorem (\Cref{thm:Hajnal-Szemeredi}) there exists a partition
\[ F' = F_1 \cup \cdots \cup F_k \]
of $F'$ into a disjoint and balanced union of $k$ independent sets, which will be processed separately in $k$ \emph{batches}. Since $\card{F'} \leq \card{F} \leq \exnd{{0.9}}n$ 
we can greedily create a candidate list $\CC$ for $F'$ with
\length
$$s \geq \frac{n-(\D+1)\card{F'}} {(\D^2+1) \cdot \card{F'}}
 \geq \frac{n} {2 \D^2 \exnd{{0.9}}n}
 \geq \expp{\D^{0.89}}
 \geq \D^{0.6} . $$
We will artificially truncate this to a candidate list $\CC$ with coverage exactly $s=\D^{0.6}$;
this is important. 
Let $\CC_i$ be the candidate list obtained by restricting to $F_i$, that is, $\CC_i = (C(v))_{v \in F_i}$;
note that each $\CC_i$ also has coverage~$s$.
(We could also have separately constructed candidate lists for each $F_i$ as needed; it makes no difference.)
Our \textsc{Bridging} strategy is as follows.
\begin{strategy}[\textsc{Bridging}]\label{strat:bridging}
For each $i \in [k]$ sequentially, Builder performs the following two subphases.
\begin{enumerate}[$(i)$]
\item\label{i:preprocessing} In the first subphase, Builder follows the \textsc{Meta}$(\CC_i,A)$-rule with $A = N_H(F_i)$ for $\Poi{\D^{0.8}n}$ rounds. 
    That is, when offered a vertex $w \in N_H(F_i)$, Builder will claim an unclaimed edge (if there is any) of the form $wv'$ with $v' \in C(v)$ for some $v \in F_i$ with $w \in N_H(v)$.
    Otherwise, Builder claims no edge. 
    
    \smallskip\noindent\emph{Note:} given $w\in N_H(F_i)$ there may be $\D^{0.1}$ choices for $v$ (no more, since $w$ is not a hub),
    and for each such $v$, there are at most $s=\D^{0.6}$ choices of $v'$.
    
\item\label{i:main} In the second subphase, Builder follows the \textsc{Meta}$(\CC_i,A)$-rule with $A = \bigcup_{v \in F_i} C(v) \cup N_H(C(v))$ for $\Poi{\D^{0.8}n}$ rounds.
    That is, if offered a vertex $v' \in C(v)$ for some $v \in F_i$, Builder will claim an unclaimed edge (if there is any) of the form $v'w$, for $w \in N_H(v)$. If offered a $w' \in N_H(C(v))$ for some $v \in F_i$, Builder will claim the edge $w'v$ if it is not already claimed. Otherwise, Builder claims no edge.
    
    \noindent\emph{Note:} given $v' \in C(v)$, $v$ is uniquely determined, and there are at most $\D$ choices of $w$.
    Also, given $w' \in N_H(C(v))$, that is, $w' \in N_H(v')$ for some $v' \in C(v)$, $v'$ is uniquely determined and so is $v$.
\end{enumerate}
After the second subphase, we execute the results for $\CC_i$, i.e., we repair all possible vertices. 
Note that this updates the labelling of $[n]$ and so might change the relevant edges. 
Then, if $i < k$, we start afresh with batch $F_{i+1}$. 
\end{strategy}

This strategy overcomes the obstacles outlined in \Cref{subsec:batches_limitations}
in the fashion summarised in \Cref{subsec:batches_workaround}.

\begin{lemma}\label{lem:bridging_analysis}
Let $G$ be a graph containing at most $\exnd{{0.95}} n $ failures and at most $\exnd{1.05} n $ hubs.
If Builder follows \Cref{strat:bridging}, then whp afterwards there are at most  $\exnd{{1.04}} n$ failures (including hubs and hub-adjacent vertices) in the resulting graph.
\end{lemma}

\begin{proof}
\Cref{fig:swaps} may be helpful. 
We focus on the failed vertices that are not hubs nor hub-adjacent.
Each such vertex $v$ belongs to a unique set $F_i\subseteq F'$ and has candidate list $C$. 
 Recall that \Cref{strat:bridging} gives coverage $s = \card{C} = \D^{0.6}$.

\noindent\textbf{First subphase.} 
Consider a failed $v \in F'$.
\begin{assert} \label{C1}
Let $E_1(w)$ be the event that a given $w \in N_H(v)$ is offered at least $s$ times, and thus makes all its incident edges.
Then $E_1(w)$ has failure probability $p_1 \leq \exnd{0.79}$.
\end{assert}
\begin{reason}
\sloppypar
By \eqref{e:Poissonisation} and the Chernoff bound (\Cref{lem:chernoff}), the failure probability is
$ p_1 
 \leq \prob{\Poi{\tau} < s}
 = \prob{\Poi{\D^{0.8}} < \D^{0.7}} 
 \leq \prob{\Poi{\D^{0.8}} < \tfrac12 \D^{0.8}} 
 \leq \expp{-\frac18 \D^{0.8}}
 \leq \exnd{{0.79}} $.
\end{reason}

\begin{assert} \label{C2}
Let $E_2(v)$ be the event that the number of $w \in N_H(v)$ failing to make all their incident edges is at most $\D^{0.3}$.
Then $E_2(v)$ has failure probability $p_2 \leq \exnd{1.08}$.
\end{assert}
\begin{reason}
By \eqref{e:Poissonisation} for each $w \in N_H(v)$ the event $E_1(w)$ fails independently of the other neighbours of $v$,
so the chance of $\D^{0.3}$ or more $E_1$-failures is 
$ p_2
 \leq \prob{\Bin{\D,p_1} \geq \D^{0.3}}$.
Taking $\mu = \D \, p_1 \leq \D \exnd{0.79}$ and  $\mu+t = \D^{0.3}$,
by the Chernoff bound (\Cref{lem:chernoff}),
\begin{align*}
p_2 \leq
 \parens{\frac{e \mu}{\mu+t}}^{\mu+t}
 \leq \parens{\frac{e \D \exnd{0.79}}{\D^{0.3}}}^{\D^{0.3}}
 \leq \parens{\exnd{0.78}}^{\D^{0.3}}
 = \exnd{1.08} .
\end{align*}
\end{reason}

\noindent\textbf{Second subphase.} Focus on a single $v' \in C(v)$.
\begin{assert} \label{C3}
Let $E_3(v')$ be the event that $v'$ is offered at least $\D^{0.3}$ times,
so that (if $E_2$ holds) these offers complete all the edges $(w,v')$ with $w\in N(v)$.
Then $E_3(v')$ has failure probability at most $\exnd{0.79}$.
\end{assert}
\begin{reason}
Analogously with $E_1(v)$, the probability is at most 
$ \prob{\Poi{\D^{0.8}} < \D^{0.3}} 
 \leq \exnd{{0.79}}
$.
\end{reason}

\begin{assert} \label{C4}
Let $E_4(v')$ be the event that every $w' \in N_H(v')$ is offered at least once, 
so that together these $w'$ can make all the relevant edges from $v$ to $N_H(v')$.
Then $E_4(v')$ has failure probability at most $\exnd{0.79}$.
\end{assert}
\begin{reason}
By \eqref{e:Poissonisation}, a given $w'$ is never offered with probability $\prob{\Poi{\D^{0.8}} = 0} = \exnd{{0.8}}$. 
By the union bound, the probability that some $w' \in N_H(v')$ is not offered is at most
$\D \exnd{{0.8}} \leq \exnd{{0.79}}$. 
\end{reason}

\begin{assert} \label{C5}
Let $E_5(v)$ be the event that there is some $v' \in C(v)$ for which events $E_3(v')$ and $E_4(v')$ both hold.
Then $E_5(v)$ has failure probability at most $\exnd{1.38}$.
\end{assert}
\begin{reason}
For $v'\in C(v)$, by Assertions \ref{C3} and \ref{C4} the combined failure probability of $E_3(v')$ and $E_4(v')$ is at most $2 \exnd{0.79} \leq \exnd{0.78}$.
By \eqref{e:Poissonisation}, the events $\{E_5(v'):v'\in C(v)\}$ are mutually independent, and $|C(v)|=s=\D^{0.6}$, so the probability that all of them fail is at most 
$\parens{\exnd{0.78}}^{\D^{0.6}} = \exnd{1.38}$.
\end{reason}

\noindent\textbf{Strategy as a whole.}

\begin{assert} \label{C6}
Let $E(v)$ be the event that $v$ is repairable by $v'$, for some $v' \in C(v)$. 
Then $E(v)$ has failure probability at most $2 \exnd{1.08}$.
\end{assert}
\begin{reason}
We have that $v$ is repairable, by some $v' \in C(v)$, unless the first-phase event $E_2(v)$ fails
or the second-phase event $E_5(v)$ fails. 
This has probability at most $\exnd{1.08} + \exnd{1.38} \leq 2 \exnd{1.08}$.
\end{reason}

\smallskip

Extend the definition of $E(v)$ also to vertices $v$ outside of $F'$, making the event true by definition in such cases. 
It remains the case that $E(v)$ fails with probability at most $ \exnd{1.07}$.
For a set of vertices $X$ at pairwise distance at least three, by \eqref{e:Poissonisation} the events $\{E(v) \colon v \in X\}$ are mutually independent.
It follows from \Cref{lemma:concentration} that the total number of failed events is whp at most $\exnd{1.06} n$.

Adding to this the hub and hub-adjacent vertices ignored in this phase, of which by hypothesis there are at most 
$|M \cup N_H(M)| 
 \leq (1+\D) \exnd{1.05} n
 \leq \tfrac12 \exnd{{1.04}} n$,
the total number of vertices failed at the end of this phase is whp at most $\exnd{1.04} n$.
\end{proof}

\section{Final phase}\label{sec:final_phase}
Starting with the set $F_1$ of failed vertices from the bridging phase, 
in this phase we will reduce the number of vertices iteratively, the $i$th iteration reducing $F_i$ to a smaller failed set $F_{i+1}$, until there are no failed vertices remaining. 
The iterations use progressively less time, so that the total time is $o_\D(\D) n$. 
Each iteration applies the idea described in \Cref{subsec:batches_limitations}.

Consider an iteration addressing a failed set~$F$, with
\begin{align} \label{e:F1}
  \card{F} \leq \card{F_1} \leq \exnd{1.04} n .
\end{align}
Choose an independent subset $I \subseteq F$ of cardinality $\card{I} = \card{F} / (\D+1)$, and candidate set $\CC$ for $I$, with coverage $s$; 
details follow shortly, but in contrast to the bridging phase, we will choose the coverage as large as possible.
Apply the following strategy, which builds each repair gadget $R(v,v')$ using offers of $v'$ and $N_H(v')$.

\begin{strategy}\label{strat:final_phase}
With $\tau=\tau(s)$,
Builder follows the \textsc{Meta}$(\CC,A)$-rule with $A = \bigcup_{v \in I} C(v) \cup N_H(C(v))$ for $\Poi{\tau n}$ rounds: If offered a vertex $v' \in C(v)$ for some $v \in I$, Builder will claim any unclaimed edge of the form $v'w$, for $w \in N_H(v)$; if offered a vertex $w' \in N_H(C(v))$ for some $v \in I$,  Builder will claim the edge $vw'$ if it is not already claimed. Otherwise, Builder claims no edge. 
After all rounds, Builder executes their result, performing a maximum set of repairs.
\end{strategy}

\begin{remark}
This procedure differs structurally a little from that in the bridging phase. 
There, \emph{one time}, we \emph{partitioned} a failed set $F$ into independent sets and addressed each in turn.
Here, we construct just \emph{one} independent set, address it to obtain a smaller failed set, and \emph{iterate}. 
In the bridging phase, the partitioning was essential: if we repeatedly chose new independent sets, a failed vertex would be in none of them with substantial probability, so we could reduce the number of failures only by a constant factor. 
In this final phase, the iteration is essential to resolve every single failed vertex.
Within an iteration, we choose to work with a single independent set, for simplicity. (Alternatively, we could have worked in turn through $\D+1$ independent sets partitioning $F$, which is more time-efficient.)
Our choice means that there is no gain in resolving more than a constant fraction of the failures in an independent set: even if we resolved absolutely every failure in one independent set, that would resolve only a $1/(\D+1)$ fraction of all failures, and resolving (for example) half this many leads to running time of the same order of magnitude.
So, reminiscent of the greedy phase, extra iterations here decrease the number of failures ``only'' exponentially; the difference is that now the iterations get faster and faster.
\end{remark}

We now give the details of \Cref{strat:final_phase}.
Greedily construct the independent set $I \subseteq F$, with $\card I = \card F / (\D+1)$,
with each vertex added to $I$ blocking at most $\D$ other vertices from consideration.
Then greedily construct a candidate set $\CC$ for $I$, respecting \Cref{d:candidatelists}, as follows.
Each vertex in $I$ blocks vertices in a 2-ball around it; since there are at most $\D^2+1$ vertices in this ball and even the initial failed set $F_1$ has tiny cardinality, this leaves more than $n/2$ vertices available as candidates.
Then, each candidate vertex added to $\CC$ blocks vertices in a 3-ball round it, with at most $\D^3+1$ vertices per ball. 
It follows that we may obtain a candidate set with 
$\left| \bigcup_{v \in I} C(v) \right| = n/2(\D^3+1)$ exactly (up to integrality).

Allocating these candidate vertices equally to the failed vertices $I$ gives a candidate list with \length 
\begin{align} \label{e:sFinal}
s 
 &= \frac{\left| \bigcup_{v \in I} C(v) \right|}{\card I}
 = \frac{n/2(\D^3+1)}{\card{F} / (\D+1)}
 = \frac{n}{2\card{F}} \frac{\D+1}{\D^3+1}
 \geq \expd{1.03} ,
\end{align}
using $\card F \leq \card{F_1} \leq \exnd{1.04} n$ and assuming that $\D$ is sufficiently large.
Let
\begin{equation}\label{e:tau}
\tau = \tau(s) \coloneqq \sqrt{\D} \cdot s^{-1/ (3\D)} .
\end{equation}

\begin{lemma}\label{lem:analysis:finalphase}
Let $G$ be a graph, $I$ an independent set of failed vertices, and $\CC$ a candidate list for $I$ of \length $s$. 
Let Builder follow \Cref{strat:final_phase} for $\Poi{\tau n}$ rounds, with $s$ and $\tau$ given by \eqref{e:sFinal} and \eqref{e:tau}.
Then, for $\D$ sufficiently large and 
for a given $v \in I$,
\[
  \prob{v\text{ is not repairable}} \leq \expp{-s^{1/3}} .
\]
Furthermore, the events that $v$ is not repairable for $v \in I$ are mutually independent.
\end{lemma}

\begin{proof}
Let $v' \in C(v)$ be a swapping candidate.  
All edges required for this swap, i.e., all edges in $R(v,v')$, are claimed if $v'$ is offered at least $\D$ times and every neighbour $w' \in N_H(v')$ is offered at least once.
By \eqref{e:Poissonisation} each vertex is offered $\Poi \tau$ times, so
\begin{equation*}
 \prob{\text{$v$ is repairable using $v'$}} 
 = \prob{\Poi{\tau} \geq \D} \cdot \prob{\Poi{\tau} \geq 1}^{|N_H(v')|} 
 \geq \frac{e^{-\tau} \tau^\D}{\D!} \cdot \parens{1 - e^{-\tau}}^{\D}.
\end{equation*}
Note that, by \eqref{e:sFinal}, \eqref{e:tau}, as $\D \to \infty$ we have that $\tau \to 0$ and so we may assume that $\D$ is sufficiently large and $\tau$ is sufficiently small. In which case, $e^{-\tau} \geq \tfrac12$, $1-e^{-\tau} \geq \tfrac{\tau}{2}$, and (by Stirling's formula) $\D! \leq 3 \sqrt{\D} \, (\D/e)^\D$.
Thus,
\begin{align} \label{eq1:lem:final_phase}
\prob{\text{$v$ is repairable using $v'$}} 
   &\geq \frac{\tfrac12 \cdot \tau^\D}{3 \sqrt{\D} \, (\D/e)^\D} \cdot \parens{\frac{\tau}{2}}^\D
   \geq \frac{1}{6\sqrt{\D}} \parens{\frac{e}{2}}^\D \parens{\frac{\tau^2}{\D}}^\D
   \geq \parens{\frac{\tau}{\sqrt \D}}^{2\D}
    = s^{-2/3} .
\end{align} 

By \eqref{e:Poissonisation} the events that $v$ is repairable using $v'$ are mutually independent for distinct $v' \in C(v)$, by \Cref{d:candidatelists} in the construction of $\CC$,
and each $v$ has $|C(v)| = s$ swapping candidates.
Hence, 
\begin{align} \label{eq:lem:final_phase}
  \prob{v \text{ is not repairable}} 
  \leq \parens{1-s^{-2/3} }^{s} 
  \leq \parens{\expp{-s^{-2/3}}}^s
  = \expp{-s^{1/3}} .
\end{align}
Finally, by \Cref{d:candidatelists}, the events that $v$ is not repairable depend on offers of distinct vertices for distinct $v \in I$ and so by \eqref{e:Poissonisation} these events are mutually independent.
\end{proof}

\begin{lemma}\label{claim:final_phase:swaps}
In every iteration of \Cref{strat:final_phase}, with parameters given by \eqref{e:F1}, \eqref{e:sFinal}, and \eqref{e:tau},
at least a $\frac{1}{2\D}$-fraction of the failed vertices $F$ are repaired, with overall failure probability $\bigo{n^{-1}}$. 
\end{lemma}

\begin{proof}
As discussed, choose an independent set $I$ of cardinality
$\ceil{\card F / (\D+1)}$.
(Integrality is plausibly relevant when $F$ becomes small.)
It suffices to repair 2/3rds of these vertices, since 
$\frac23 \card I \geq \frac23 \frac{\card F}{\D+1} \geq \frac{1}{2\D} \card F$ for $\D \geq 3$.
Execute \Cref{strat:final_phase} and
let $I'$ denote the subset of $I$ \emph{not} repaired;
failure requires $\card{I'} > \card I / 3$.

For $|I| <  36 \ln n$, it holds by \eqref{e:sFinal} 
that $s \geq \frac{n}{\poly(\D) \card{F}} \geq \sqrt n$.
Thus for any $v \in I$, 
by \Cref{lem:analysis:finalphase}, 
\begin{align*}
\prob{v\text{ is not repairable}} 
   \leq \expp{-s^{1/3}} 
   = \expp{-n^{1/6}} 
   = o(n^{-2}) .
\end{align*}
Since there are only $n$ vertices, by the union bound, the probability of any failure of this sort is at most $n^{-1}$. In particular, whp the process will end once the independent sets get this small.

For  $|I| \geq  36 \ln n$, by \Cref{lem:analysis:finalphase}, $|I'|$ is dominated by a $\Bin{|I|, \expp{-s^{1/3}}}$ random variable.
Since by \eqref{e:sFinal} $s \geq \exnd{1.03}$, 
$|I'|$ is dominated by a $\Bin{|I|, 1/6}$ random variable.
This has mean $\mu =\card I / 6$ and
failure implies that $\card{I'}$ is at least twice this value.
By the Chernoff bounds (\Cref{lem:chernoff}),
\begin{align*}
  \prob{|I'| \geq \frac{|I|}{3}} 
   \leq \expp{-\frac \mu 3} 
   = \expp{-\frac{\card I}{18}}
   \leq \expp{-2 \ln n}
   = n^{-2} .
\end{align*}
Finally, since each successful iteration repairs at least one vertex, there at most $n$ successful iterations. Hence, by a union bound the chance of any failure of this sort is at most $n^{-1}$.
\end{proof}

In what follows we will assume that the conclusion of \Cref{claim:final_phase:swaps} holds deterministically.

\begin{lemma} \label{lem:finalTotalTime}
With iteration $i$ of \Cref{strat:final_phase} running for time $\Poi{\tau_i n}$, the total time for the final phase is $\Poi{\tau n}$,
with $\tau \leq \expp{-\tfrac14 \D^{0.01}} = o_\D(1)$.
\end{lemma}

\begin{proof}
By \eqref{e:sFinal} and \eqref{e:tau},
\begin{align*}
 \tau_1
  &= \sqrt \D \exnd{1.03}^{-1/3\D} 
  = \sqrt \D \expp{-\frac13 \D^{0.03}} .
\end{align*}
Also from \eqref{e:sFinal} and \eqref{e:tau}, using \Cref{claim:final_phase:swaps}
\begin{align*}
  \frac{\tau_{i+1}}{\tau_i} 
    &= \parens{\frac{s_{i+1}}{s_i}}^{-1/3\D} 
    \leq \parens{1-\frac{1}{2\D}}^{1/3\D} .
\end{align*} 
Then, 
\begin{align*}
 \tau 
  &\leq \sum_{i=1}^\infty \tau_i 
  \leq \tau_1 \sum_{i =0}^\infty \parens{1-\frac{1}{2\D}}^{i/3\D} 
  \leq  \tau_1 \sum_{i =0}^\infty \parens{1-\frac{1}{2\D}}^{\lfloor i/3\D \rfloor}
  \leq   3 \D  \cdot \tau_1 \sum_{j =0}^{\infty} \parens{1-\frac{1}{2\D}}^{j} 
  \\& = (3\D)(2\D) \tau_1 
   = 6 \D^{5/2} \expp{-\frac13 \D^{0.03}}
  \leq \exnd{{0.02}} 
  = o_\D(1).
\end{align*} 
\end{proof}

\section{Proof of Theorem \ref{thm:main:star}} \label{sec:proof_main}
In this section we recapitulate Builder's strategy and prove our main result. 

\begin{proof}[Proof of \Cref{thm:main:star}] Starting with the empty graph,
Builder first follows the \textsc{Greedy} strategy (Strategies \ref{strategy:greedy} and \ref{strategy:extragreedy}), described in \Cref{sec:greedy}. By \Cref{lem:greedy_analysis,lem:extragreedy_analysis} it holds  that \whp{} after the \textsc{Greedy} phase there are at most $\exnd{{0.95}}n$ failed vertices and at most $\exnd{{1.05}}n$ hubs. 

Builder then follows the \textsc{Bridging} strategy (\Cref{strat:bridging}), described in \Cref{sec:bridging_phase}. By \Cref{lem:bridging_analysis}, whp after the bridging phase there are at most $\exnd{{1.04}}n$ failed vertices. 

Finally, Builder iteratively follows \Cref{strat:final_phase}, as described in \Cref{sec:final_phase}, iterating until no failed vertices remain.
By \Cref{claim:final_phase:swaps}, whp
every iteration succeeds in repairing at least a $\tfrac1{2\D}$ fraction of the failed vertices.
It follows that whp all three phases succeed. Also note that upon termination of the \textsc{Final}  phase, Builder has succeeded in constructing a graph that contains a copy of $H$. 

The \textsc{Greedy} phase takes a total of  
$\Poi{\tau_{\textsc{Out}} \cdot n}+\Poi{\tau_{\textsc{In}} \cdot n}$ rounds, 
where $\tau_{\textsc{Out}} = \frac{\D}{2} + \D^{0.99}$   
and $\tau_{\textsc{In}} = \D^{0.99}$. 
The \textsc{Bridging} phase uses $\Poi{\tau_B \cdot n}$ rounds, with $\tau_B = 2 \D^{0.8}$, for each of $\D^{0.1}$ batches, 
for a total of $\Poi{3\D^{0.9}} n$ rounds.
Finally, by \Cref{lem:finalTotalTime}, whp the \textsc{Final} phase uses $\Poi{\tau_F \cdot n}$ rounds, with $\tau_F \leq \expp{-\tfrac14 \D^{0.01}}$.

Hence, altogether, whp the number of rounds Builder uses is at most 
\[
\Poi{ \left(\frac\D2 + 2\D^{0.99} + 3\D^{0.9} +\expp{-\tfrac14 \D^{0.01}} \right)\cdot n}
 = \Poi{\parens{1 + o_\D(1)}\frac{\D n}{2}}.
\]

The above together with \Cref{PoEnough} complete the proof of \Cref{thm:main:star}.
\end{proof}

\section{The semirandom tree process}\label{sec:proof_main_tree}
In this section we describe how Builder's strategy can be adjusted to work in the semirandom \text{tree} process. While the main ideas used to derive \Cref{thm:main:star} are also applicable here, there are some technical difficulties we highlight in this section. 

\subsection{General probabilistic tools and results}
In this subsection, we introduce some probabilistic tools necessary to analyse the edge distribution of a random spanning tree and use them to derive some key properties about uniform spanning trees. To sample a spanning tree uniformly at random, we rely on the  \emph{Aldous-Broder algorithm} \cite{Al1990,Br1989}.
Roughly, this process performs a random walk on $G$ (in our case, $G=K_n$) and only records the edges which lead to vertices not previously visited.
Specifically, begin with an arbitrary vertex $v_1 \in V(G)$ and perform a simple random walk $(X_m)_{m \geq 1}$ on the graph $G$ with initial vertex $X_1=v_1$. For each vertex $v \in V(G)\setminus \{v_1\}$ let $t(w) =\inf \{ m \in \N \colon X_m=v\}$ be the hitting time for $v$. Note that $t(v)$ is almost surely finite for each $v$. We construct a random tree $T \subseteq G$ by letting the parent of each $v \in V(G)\setminus \{v_1\}$ be the vertex $X_{t(v) -1}$ immediately preceding the first time we visit~$v$. 

\begin{lemma}[\cite{Al1990,Br1989}]\label{lem:aldousbroder}
Let $G = (V,E)$ be a graph. If $T$ is constructed according to the Aldous-Broder algorithm on $G$, then $T$ is uniformly distributed over all spanning trees of $G$.
\end{lemma}

We also rely on a result due to Moon \cite{Mo1967}, appearing for example in {\cite[p. 205]{AiZi2010}}. A \emph{rooted forest} is a forest together with a choice of a root in each component. We can imagine a rooted forest as a directed graph, by orienting edges away from the root in each component. This defines a subgraph relation on the set of rooted forests on $n$ vertices.
\begin{lemma}[{\cite{AiZi2010}}]\label{lem:rooted_forest_count}
    Let $F_{n,m}$ be a rooted forest with vertex set $[n]$ having $m$ components. Then, the number of rooted trees with vertex set $[n]$ containing $F_{n,m}$ is exactly $n^{m-1}$.
\end{lemma}

Using the Aldous-Broder algorithm and \Cref{lem:rooted_forest_count} we can derive the following two key lemmas about the edge distribution of a random spanning tree. 

The first governs the probability of a \emph{wasted} round, in which Builder is offered no edges from a predefined set of \emph{relevant} edges.

\begin{lemma}\label{lem:failurebound:strat1}
Let $\usefulE{} \subseteq \binom{[n]}{2}$ be a set of edges and $T$ a tree on vertex set $[n]$ sampled uniformly at random from all spanning trees of $K_n$. Then,
\begin{align*}
    \prob{S \cap E(T) = \emptyset} \leq \exp\left(- \, \frac{|\usefulE{}|}{n}\right).
\end{align*}
\end{lemma}

\begin{proof} 
We may assume by \Cref{lem:aldousbroder} 
that $T$ is sampled according to the Aldous-Broder algorithm run on $K_n$ starting from $v_1:=1$.
Suppose the vertices are added to the tree $T$ in the order $v_1,v_2,\ldots, v_n$.

Direct each edge in $\set{v_i, v_j} \in S$ backwards according to this order, i.e., 
directed from $v_i$ to $v_j$, as $(v_i,v_j)$, if $j<i$.
Denote by $N_S^+(v_i)$ the set of edges out of $v_i$,
with $d_S^+(v_i) = \card{N_S^+(v_i)}$ its outdegree;
note that $\sum_i d_S^+(v_i) = \card S$ since each edge gets directed.

We claim that, conditioned on the order, independently for each pair $0<j<i\leq n$, 
$\prob{\parent{(v_i)}=v_j} = f(i,j) \geq 1/n$, for some fixed formula~$f$.
Consider the step after $v_\im$ has been added to the tree. 
This step walks to a random vertex $w$.
It is convenient to imagine that the $K_n$ contains self-loops, so that there are $n$ possible outcomes for $w$ including $w=v_\im$; this makes no difference to the Aldous-Broder algorithm's outcome.
There are two cases. 
The first is that $w$ is a new vertex, which thus becomes $v_i$, so $\parent(v_i)=v_\im$;
this occurs with probability $(n-(\im))/n$.
The second is that $w$ is a revisit, equally likely to be any of $v_1$ through $v_\im$; this occurs with probability $(\im)/n$.
In the second case case, by symmetry, any of $v_1$ through $v_\im$ is equally likely to become $\parent(v_i)$, giving probability mass $\frac{\im}{n} \cdot \frac1\im = \frac1n$ to each $v_j$ with $j<i$.
Vertex $v_\im$ has additional probability mass from the first case (and it is easy to see that the probabilities add to 1 as they must).
This defines $f(i,j)$ and establishes the claim. 

Any directed edge in $N_S^+(v_i)$ goes to a lower-index endpoint $v_j$,
and by the claim, $\prob{\parent(v_i)=v_j} \geq 1/n$.
For a given $v_i$ these events are disjoint, 
so $\prob{ (v_i, \parent(v_i)) \in S } \geq d_S^+(v_i) /n$. 
By the claim, these events are independent for distinct $v_i$, so
\begin{align*}
\prob{S \cap E(T) = \emptyset} 
  &\leq \prod_i \parens{1- \frac{d_S^+(v_i)}{n}}
  \leq \prod_i \expp{-\frac{d_S^+(v_i)}{n}}
  = \expp{- \sum_i \frac{d_S^+(v_i)}{n}}
  = \expp{-\frac{\card S}{n}} .
\end{align*}
\end{proof}

In one round, by symmetry (see \eqref{oneEdge}),
the probability that Builder is \emph{offered} any given edge is $\frac{2}{n}$,
but this does not mean that Builder will \emph{claim} this edge, since several edges of interest may have been offered. 
The next lemma shows that whenever Builder is interested in a set $S$ of ``not too many'' edges, 
each of those edges is claimed
with probability at least $\frac{1}{n}$.

\begin{lemma}\label{lem:failurebound:strat2}
Let $S\subseteq \binom{[n]}{2}$ be a set of edges satisfying $\card S \leq \frac{n}{4}$, let $e \in S$ and let $T$ be a tree with vertex set $[n]$ sampled uniformly at random  from all spanning trees of $K_n$.
Then, 
\begin{align*}
  \prob{S \cap E(T) = \{ e \} } \geq \frac{1}{n}.
\end{align*}
\end{lemma}

\begin{proof}
Since each spanning tree of $K_n$ contains $n-1$ edges, by symmetry it is clear that for each $e \in S$,
\begin{align}
  \prob{e \in E(T)}
   &= \frac{n-1}{\binom n 2}
   = \frac 2 n . \label{oneEdge}
\end{align}

Then it suffices to show that the probability that $T$ contains any fixed pair of edges is at most $\frac{4}{n^2}$, for if this holds, then for any $e\in S$ we have 
\begin{align*}
    \prob{S \cap E(T) = \{e\}} \geq \prob{e \in E(T)} - \sum_{f \in S \setminus\{ e\}} \prob{\{e,f\} \subseteq E(T)}  \geq \frac{2}{n}- (|S|-1)\frac{4}{n^2}  \geq \frac{1}{n}.
\end{align*}

Now, let $e,f \in S$ be an arbitrary pair of edges and assume first that they are adjacent. If we view this pair of edges as a forest (with $n-3$ further isolated vertices) on $[n]$, then there are precisely $3$ different rooted forests $F_1,F_2,F_3$ on the edge set $\set{e,f}$, since the roots of isolated vertices are fixed (recall that a rooted forest is a forest with a choice of a root in each component). By \Cref{lem:rooted_forest_count}, for each $F_i$ there are $n^{n-5}$ rooted trees which contain $F_i$. Since each tree can be rooted in $n$ different ways, it follows that the total number of spanning trees containing $\{e,f\}$ is $3n^{n-4}$.
If $e$ and $f$ are not adjacent, then there are four possible rooted forests with edge set $\{e,f\}$, and hence by the same argument there are $4n^{n-4}$ spanning trees containing $\{e,f\}$. 

In both cases, by Cayley's formula, the probability that $\set{e,f}$ is contained in a uniformly chosen spanning tree is at most $\frac{4n^{n-4}}{n^{n-2}} = \frac{4}{n^2}$, as claimed.
\end{proof} 

\subsection{Proof of Theorem \ref{thm:main:tree}}

We follow a strategy with the same three phases as that for the semirandom star.
In the \textsc{Greedy} phase, there is now no need to consider directed edges, and the two subphases described in \Cref{sec:greedy} can be combined into one, in which in each round Builder attempts to claim a failed edge.

\begin{lemma}
After $\Poi{(\frac12 \D + 3 \D^{0.9}) n}$ rounds of a greedy strategy, whp the number of failed vertices is at most $2 \exnd{0.9} n$ and the number of failure hubs at most $\frac12 \exnd{1.07} n$.
\end{lemma}

\begin{proof}
The first $\Poi{(\frac12 \D + 2 \D^{0.99}) n}$ rounds whp include
$(\frac12 \D +  \D^{0.99}) n$ rounds,
and we analyse the greedy strategy in three subphases, of durations
${\frac12 \D n}$, ${\D^{0.99} n}$, and $\Poi{\D^{0.99} n}$.

We first claim that, after the first subphase, whp the number of failed edges is less than $\ell_1 := \D^{1/2} n$.
If not, by \Cref{lem:failurebound:strat1}, 
the number $X$ of wasted rounds (and thus edges failed at the end) is at most 
$\Bin{\frac12 \D n, \exnd{1/2}}$.
By the Chernoff bounds (\Cref{lem:chernoff}), 
$$\prob{X \geq \ell_!} 
 \leq \parens{\frac{e \cdot \frac12 \D \exnd{1/2} n}{\frac12 \D^{1/2} n}}^{\frac12 \D^{1/2} n}
 < \parens{\frac12}^{\frac12 \D^{1/2} n} = o(1) . $$
So, whp $X < \ell_1$, contradicting the hypothesis that the number of failed edges remains at least $\ell_1$.

Similarly, suppose at the start of the second subphase there are fewer than $\D^{1/2} n$ failed edges, but after the second subphase there are more than $\ell_2:=n/6$ failed edges.
Then by \Cref{lem:failurebound:strat1}, in each round Builder is offered a failed edge with probability at least $1-\expp{-\ell_2/n} \geq 1-\expp{-1/6} > 1/10$, and hence the total number of rounds in this subphase in which Builder claims a failed edge is at least $\Bin{\D^{0.99}n,1/10}$. 
Whp this exceeds $\D^{1/2}n$, contradicting the assumption that there were fewer failed edges at the start of this subphase.

Then the third subphase starts with a set $S$ of failed edges, with $\card S \leq \ell_2 = n/6$.
By \Cref{lem:failurebound:strat2}, in each round, for any $e \in S$, with probability at least $1/n$, the tree $T$ offered intersects $S$ uniquely in~$e$.  Only accept an edge in such rounds and do not update $S$. Should an accepted edge be the unique one again later, ``accept'' it again. 

Since at most one edge can be accepted in each round, by a similar argument as in \Cref{s:Poissonisation}, if we let $X_e$ be the number of times that an edge $e \in S$ is accepted during this third subphase, we have that the random vector $(X_e \colon e \in S)$ is distributed as a mutually independent collection of Poisson random variables, each of whose mean is at least $\D^{0.99}$, and so each edge fails to be accepted with probability at most $\exnd{0.99}$.

\begin{sloppypar}
In particular, the number of edge failures at the end of this subphase is dominated by $\Bin{n/6, \exnd{0.99}}$, and by Chernoff is whp at most $\exnd{0.99} n$.
\end{sloppypar}

For a vertex $v$ to be a hub after the third subphase means that $\D^{0.1}$ vertices in $N(v)$ are failed, and thus at least $\frac12 \D^{0.1}$ edges in $G[N^2(v)]$ are failed.
As just argued, any such edge failed at the start of the subphase remains failed with probability at most $\exnd{0.99}$, and edge failures are mutually independent, so the probability that a given $v$ is a hub is at most 
$$
\prob{\Bin{\D^2, \exnd{0.99}} \geq \D^{0.1}}
 \leq \parens{\frac{e \D^2 \exnd{0.99}}{\frac12 \D^{0.1}} }^{\D^{0.1}}
 \leq \parens{e^{-\D^{0.98}}}^{\D^{0.1}}
 = \exnd{1.08} .
$$
Hence, the expected number of hubs is at most $\exnd{1.08} n$.
Furhermore, since edge failures are mutually independent, and ``hub-ness'' of vertices at distance five or more depends on failures in disjoint edge sets, by \Cref{lemma:concentration} the number of hubs is whp at most $\exnd{1.07} n$.
\end{proof}

In the \textsc{Bridging} phase here,
the preprocessing stage needed for the semirandom star process is not necessary, due to the independence of offered edges in the tree model. As we will see below, it is sufficient that (as for the star process) we exclude the hubs and so can reduce the number of partition classes and increase the number of rounds spent on each class. 

In both the bridging and final phases, the set of relevant edges is given as $S := \bigcup_{v \in I}\bigcup_{v' \in C(v)} R(v, v')$, where $I \subseteq F$ is an independent set, $\CC$ is a candidate list, and the set of swapping candidates $C(v) \in \CC$ and the edges $R(v,v')$ necessary for a swap are defined as in \Cref{sec:substitutions}. We observe that for all $v \in I$  and for all $v' \in C(v)$ it holds that $|R(v,v')| \leq 2\D$ and $\CC$ can be greedily chosen such that $|C(v)| = \frac{n}{(\D^2+1)|I|}$, 
implying that 
\begin{align}
 \card S
 &=
  \sum_{v \in I} \sum_{v' \in C(v)} |R(v, v')| 
  \leq |I| \cdot \frac{n}{(\D^2+1)|I|} \cdot 2\D  
  \leq \frac{n}{4}.  \label{TreeSmallRelevant}
\end{align}
In these subphases, as in the third subphase of the tree-process greedy phase, in a round where a tree $T$ is offered, we will  accept an edge $e$ if and only if $S \cap E(T) = \set e$, and we do not update $S$ (except of course when we switch to the next independent set $I$).
By \eqref{TreeSmallRelevant} and \Cref{lem:failurebound:strat2}, in each round, Builder claims a given $e \in S$ with probability at least $\frac{1}{n}$ and so, by a similar argument as before the probability that an edge $e \in S$ is claimed after $\Poi{\tau n}$ is at least $1- \expp{-\tau}$ and these events are mutually independent for different edges. Therefore, given a failed vertex $v$ and a swapping candidate $v' \in C(v)$, the probability that $v'$ can repair $v$ after $\tau$ rounds is at least
\begin{align}\label{e:newprob}
  \prob{R(v,v') \subseteq G} \geq (1-\expp{-\tau})^{2\D},
\end{align}  
and by extension,
\begin{align}\label{e:repairprob}
 \prob{v \text{ is not repairable}} \leq (1-(1-\expp{-\tau})^{2\D})^{|C(v)|}.
\end{align}

In the \textsc{Bridging} phase, the strategy is this.
Partition the non-hub failed vertices into $\D^{0.1}$ independent sets of equal size,
for each independent set create a candidate list of \length $\D^{0.6}$,
and spend time $\Poi{\tau n}$, with $\tau=\D^{0.8}$, on repairing each independent set, for time $\Poi{\D^{0.9} n}$ in all.
(For the star process we also avoided repairing hub-adjacent vertices, to avoid overloading their hub neighbours with relevant edges; consider vertex $w$ in \Cref{fig:swaps}.
In the tree process all that matters is that the gadgets' edges are distinct; it is irrelevant if many such edges have a common vertex.)

\begin{lemma}\label{lem:bridging_analysis_tree}
Let $G$ be a graph containing at most $2 \exnd{{0.9}} n $ failures and at most $\frac12 \exnd{1.07} n$ hubs.
If Builder follows the strategy described, then whp afterwards there are at most  $\exnd{{1.06}} n$ failures (including hubs).
\end{lemma}

\begin{proof}
Note that
\begin{align*}
   (1-\expp{-\tau})^{2\D} \geq 1 - 2\D \exnd{{0.8}} \geq 1 - \exnd{{0.7}} .
\end{align*}
Hence, \eqref{e:repairprob} gives that
\begin{align*}
\prob{v \text{ is not repairable}} 
 \leq \parens{\exnd{{0.7}}}^{\D^{0.6}} 
 = \exnd{{1.3}} .
\end{align*}
By \Cref{lemma:concentration}, whp the total number of such failures is at most $\exnd{1.2} n$. 
Adding to this the number of hubs, the total number of failed vertices is at most $\exnd{1.06} n$.
\end{proof}

In the \textsc{Final} phase, $\tau = o_{\D}(1)$, and hence by \eqref{e:newprob}
\begin{align*}
  \prob{R(v,v') \subseteq G} 
    \geq (1-\expp{-\tau})^{2\D} 
    \geq (\tau/2)^{2\D} ,
\end{align*} 
which is larger than in the semirandom star process (compare with \eqref{eq1:lem:final_phase}). 
Choosing the \length of the candidate lists here to match what was done for the star process, the equivalent of \Cref{lem:analysis:finalphase} therefore holds here too, 
and so do the equivalents for the rest of the final phase, 
\Cref{claim:final_phase:swaps,lem:finalTotalTime}.

Finally, the equivalent of \Cref{sec:proof_main} assembles the results from these three phases to establish \Cref{thm:main:tree}.

\section{Discussion}\label{sec:Discussion}
In \Cref{thm:main:star,thm:main:tree} we showed that in both the semirandom star process and the semirandom tree process, Builder can construct a copy of any $n$-vertex graph $H$ with maximum degree $\D=\D(n)$ in time asymptotically $\frac12 \D n$ (the number of edges potentially required),

In contrast, for other graph properties it is known that the two models behave differently.
In the semirandom tree model, properties including $k$-connectedness and Hamiltonicity can be achieved in time
asymptotically equal to the number of edges \cite{BuLi2023}, whereas in the semirandom star process this is impossible \cite{BeHeKrPaShSt2020}. 
Specifically, in the semirandom star process, \whp{} constructing a perfect matching requires at least $0.932 n$ rounds \cite{GaMaPr2022PM} and constructing a Hamilton cycle at least $1.26575 n$ rounds,
whereas in the semirandom tree process, whp a perfect matching can be constructed in $n/2 + o(n)$ rounds and a Hamiltonian cycle in $n + o(n)$ rounds \cite{BuLi2023}.

It would be interesting to know if, contrary to the above examples, there is any graph property which is \emph{easier} to achieve in the semirandom star process than in the semirandom tree process.
\begin{question}\label{q:comparetree}
Is there a graph property $\mathcal{P}$ 
that can be achieved whp in the semirandom star process
asymptotically faster than it can in the semirandom tree process?
\end{question}

What about semirandom processes in which Builder is offered the edge set of a different spanning structure, such as a random Hamiltonian path
or a random perfect matching?

Viewed in a more general context, as a semirandomised version of a ``picker-chooser'' type game (see, e.g., \cite{B02}), there is a natural generalisation of the semirandom tree process to arbitrary \emph{matroids}.
(See, e.g., \cite{O92} for an introduction to the theory of matroids.)
For example, one could consider the following semirandom process, which we call the \emph{semirandom basis process}, defined over the \emph{finite field model}.
For $q$ a prime power, in each round Builder is \emph{offered} a uniformly random basis $\{\bm{v}_1,\ldots, \bm{v}_n\}$ of the field $\mathbb{F}^n_q$, 
from which Builder may \emph{choose} one vector.

As with random graphs, threshold phenomena have been studied in the finite field model \cite{CG21,KLTY23}, which has been suggested as an interesting ``toy model'' for problems in additive number theory \cite{G05,P24,W15}. It would be interesting to know if there are strategies in the semirandom basis process which whp can build certain structures,  for example, arithmetic progressions or Schur triples, much quicker than the thresholds for their appearance in the corresponding ``classical'' random finite field process, where a sequence of elements $\bm{v_1},\bm{v}_2, \ldots \in \mathbb{F}_q^n \setminus \{ \bm{0}\}$ is generated by choosing each $\bm{v}_i$ uniformly at random from $\mathbb{F}_q^n \setminus \{ \bm{0},\bm{v}_1, \ldots, \bm{v}_\im\}$.

\begin{question}
Let $q$ be a prime power. Does Builder have a strategy in the semirandom basis process over $\mathbb{F}_q^n$ to construct
\begin{enumerate}[(a)]
\item a $3$-term arithmetic progression; or
\item a set $\{\bm{x},\bm{y},\bm{z}\}$ such that $\bm{x} + \bm{y} = \bm{z}$,
\end{enumerate}
in significantly fewer than $q^{\frac{n}{3}}$ rounds?
\end{question}

The question is of interest because it follows from work of Chen and Greenhill \cite{CG21} that in the classical random finite field process, whp both of these structures first appear when $\Theta\left(q^{\frac{n}{3}}\right)$ elements have been chosen.

\vspace{1cm}
\subsection*{Acknowledgements}
This work was initiated at the kick-off workshop of the international project ``Sparse random combinatorial structures'' (DFG CO 646/6-1 and FWF I6502). This research was funded in part by the Austrian Science Fund (FWF) [10.55776/P36131, 10.55776/I6502, 10.55776/ESP3863424], and by the European Union’s Horizon 2020 research and innovation programme under the Marie Sk\l{}odowska-Curie grant agreement No.\ 101034413 \includegraphics[width=4.5mm, height=3mm]{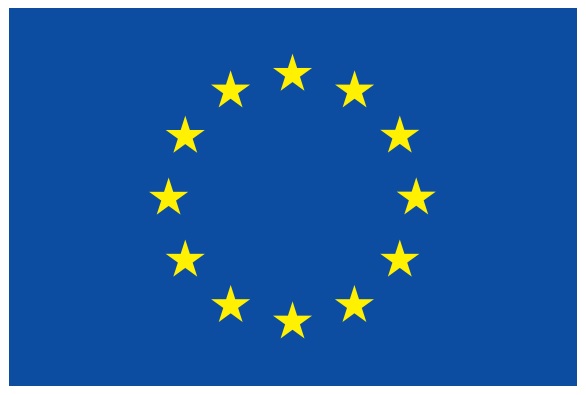}.
For open access purposes, the authors have applied a CC BY public copyright licence to any author accepted manuscript version arising from this submission.

\printbibliography
\end{document}